\documentclass[final,leqno,letterpaper]{etna}

\hypersetup{%
    pdftitle={Unilevel blocking},
    pdfauthor={Isabella Furci},
    pdfkeywords={a keyword, another keyword, and another one}
    }

\title{Block structured matrix-sequences and their spectral and singular value canonical distributions: a general theory\thanks{All the authors are members of “Gruppo Nazionale per il Calcolo Scientifico" (INdAM-GNCS). The work is partially supported by INdAM - GNCS Project “Analisi e applicazioni di matrici strutturate (a blocchi)"  CUP E53C23001670001 and  by the European High-Performance Computing Joint Undertaking  (JU) under grant agreement No 955701. The JU receives support from the European Union’s Horizon 2020 research and innovation programme and Belgium, France, Germany, Switzerland. The work of Andrea Adriani is supported by MUR Excellence Department Project MatMod@TOV awarded to the Department of Mathematics, University of Rome Tor Vergata, CUP E83C23000330006.
The work of Isabella Furci is supported by $\#$NEXTGENERATIONEU (NGEU) and funded by the Ministry of University and Research (MUR), National Recovery and Resilience Plan (NRRP), project MNESYS (PE0000006) - A Multiscale integrated approach to the study of the nervous system in health and disease (DN. 1553 11.10.2022). Furthermore, Stefano Serra-Capizzano is grateful for the support of the Laboratory of Theory, Economics and Systems – Department of Computer Science at Athens University of Economics and Business and is grateful to
the "Como Lake center for AstroPhysics” of Insubria University.
}}

\author{Isabella Furci\footnotemark[1]
        \and Andrea Adriani\footnotemark[2]
        \and Stefano Serra-Capizzano\footnotemark[3]}

\shorttitle{Block structures: spectral and singular value  distributions}
\shortauthor{I.~Furci AND A.~Adriani AND S. Serra--Capizzano}

\usepackage{braket,amsfonts}
\usepackage{amssymb}
\usepackage{array}

\usepackage{mathdots}
\usepackage{graphicx}
\usepackage{multirow}

	\usepackage[font=footnotesize]{caption}
 \usepackage[labelformat=simple]{subcaption}

\usepackage{algorithmic}
\usepackage{algorithm}
\ifpdf
  \DeclareGraphicsExtensions{.eps,.pdf,.png,.jpg}
\else
  \DeclareGraphicsExtensions{.eps}
\fi

\usepackage{array}
\usepackage{tikz}
\usetikzlibrary{matrix, calc}
\usetikzlibrary{fit}

\usepackage{pgfplots}

\usepackage{amsopn}

\usepackage{algorithmic}

\usepackage{graphicx,epstopdf}

\usepackage{amsopn}

\newcommand{\E}{\mathrm{e}}
\newcommand{\dd}{\!\mathrm{d}}

\usepackage{xspace}
\usepackage{bold-extra}
\usepackage[most]{tcolorbox}

\begin{document}

\maketitle

\renewcommand{\thefootnote}{\fnsymbol{footnote}}
\footnotetext[1]{Department of Mathematics, University of Genoa, Genoa, Italy.}
\footnotetext[2]{Department of Mathematics, Tor Vergata University of Rome,  Italy.}
\footnotetext[3]{Department of Information Technology, Uppsala University, Uppsala, Sweden. Department of Science and High Technology. University of Insubria, Como, Italy.}

\begin{abstract}
In recent years more and more involved block structures appeared in the literature in the context of numerical approximations of complex infinite dimensional operators modeling real-world applications. In various settings, thanks the theory of generalized locally Toeplitz matrix-sequences, the asymptotic distributional analysis is well understood, but a general theory is missing when general block structures are involved. The central part of the current work deals with such a delicate generalization when blocks are of (block) unilevel Toeplitz type, starting from a problem of recovery with missing data. Visualizations, numerical tests, and few open problems are presented and critically discussed.
\end{abstract}



\section{Introduction}

When dealing with the approximation of systems of ordinary or partial differential equations (DEs or PDEs), when considering restoration of signal or images with incomplete data, when approaching the inpainting problem, block structured matrices are encountered and often the single blocks have structure themselves of (block) multilevel Toeplitz type or even more general when there is no space invariance; see \cite{GLT-blocks-d-dim,GLT-blocks-1-dim,MR3689933,MR3543002,dumb,tom,MR3904142}. In fact, when dealing with DEs and PDEs with variable coefficients the resulting associated matrix-sequences belong to block multilevel generalized locally Toeplitz $*$-algebra, where the size of the blocks is usually dictated by the used approximation scheme and by the dimensionality $d$ of the physical domain, while the number of levels is exactly $d$. The reader is referred to see \cite{GLT-blocks-d-dim,GLT-blocks-1-dim,MR3543002,GSI,GSII,tom} for the theory and the treatment in detail of few emblematic applications. We also refer to the references therein and specifically to \cite{EMI,MR4389580,MR4623368,MR3689933,MR4284081,dumb,MR3904142} for the development of numerical algorithms based on the theory and to \cite{block-ext1,block-ext2, block-ext3,block-ext4,block-ext5,block-ext6,block-ext7,block-ext8} for a wide range of context in which block structured matrices are encountered and dealt using different approaches.

In this work we focus our attention on the spectral distribution and on the singular value distribution in the Weyl sense in the block unilevel setting, with rectangular blocks of different dimensions. The presence of zeros in the distribution function and the order of the zeros give information on the conditioning and on the size and nature of the subspaces where the ill-conditioning arise (see e.g. \cite{MR1705731} and references therein). In turn the latter information is relevant for designing preconditioners in Krylov methods or projectors in multigrid methods with the goal of optimal solvers, having convergence rate independent of the matrix order of the coefficient matrices. Hence such results are also of practical and computational interest.

First, we give theoretical tools for studying the distributions - in the singular and eigenvalue sense - of general matrix-sequences. Subsequently, the new tools are employed in the context of block matrix-sequences, with the matrix-sequences of the blocks having themselves a block structure under various conditions, but all concerning block unilevel Toeplitz structures.

The paper is organized as follows. In Section \ref{sec:general_block} we deal with the problem setting and with general theoretical tools.
Section \ref{sec:Toeplitz_block} concerns the use of the new tools in the case where the blocks are (variations) of Toeplitz structures equipped with Lebesgue integrable generating functions.  Section \ref{sec:num} contains numerical tests and visualizations corroborating the theoretical findings, while conclusions and the discussion of few open problems are reported in Section \ref{sec:fin}.

\section{Problem setting and general theoretical tools}
\label{sec:general_block}

In the current section we fix two positive integers $s,t$ and we introduce a particular class of structured matrices of size $sn\times tn$. Let $\nu$ and $n_1,n_2,\dots,n_{\nu}$ be positive integers such that $n= \sum_{i=1}^{\nu} n_i$. We consider the structured matrix $A_n=[(A_n)_{ij}]_{i,j=1}^{\nu}$ given by

\begin{small}
\begin{equation}
\label{eq:A_general}
A_n=\left[
\begin{tikzpicture}[baseline=(m.center)]
\matrix (m) [matrix of math nodes,column sep=0.15em,row sep=0.15em] {
|[draw,dashed, minimum width=0.8cm, minimum height=0.8cm]| A_{11} & |[draw,dashed, minimum width=2cm, minimum height=0.8cm]| A_{12} & \cdots & |[draw,dashed, minimum width=1.2cm, minimum height=0.8cm]| A_{1\nu} \\
|[draw,dashed, minimum width=0.8cm, minimum height=2.4cm]| A_{21} & |[draw,dashed, minimum width=2 cm, minimum height=2.4cm]| A_{22} & \cdots & |[draw,dashed, minimum width=1.2cm, minimum height=2.4cm]| A_{2\nu} \\
\vdots & \vdots & \ddots & \vdots \\
|[draw,dashed, minimum width=0.8cm, minimum height=1.6cm]| A_{\nu 1} & |[draw,dashed,minimum width=2cm, minimum height=1.6cm]| A_{\nu2} & \cdots & |[draw,dashed, minimum width=1.2cm, minimum height=1.6cm]| A_{\nu\nu} \\
};
\end{tikzpicture}
\right].
\end{equation}
\end{small}
Each diagonal block $A_{ii}=(A_n)_{ii}$ is a matrix of size $sn_i\times tn_i$, for $i=1,\dots,\nu$, and the off-diagonal blocks $A_{ij}=(A_n)_{ij}$, $i\neq j$, are general $sn_{i}\times tn_{j}$ rectangular matrices. For $s=t$, the diagonal blocks $A_{ii}$, $i=1,\ldots,\nu$, are square matrices.

We consider  matrix-sequence $\{A_n\}_n$ of increasing size where, for a fixed $n$, each $A_n$ is of the form in (\ref{eq:A_general}). In this setting we are interested in  the following natural assumptions
\begin{enumerate}
\item the number of blocks is $\nu^2$, $(A_n)_{ij}$, $i,j=1,\dots,\nu$, and the parameter $\nu\ge 2$ is fixed independent of $n$ (the case $\nu=1$ is not of interest for obvious reasons);
\item $\lim_{n,n_i\rightarrow \infty}=\frac{n_i}{n}=c_i\in (0,1)$, $i=1,\dots,\nu$,  $c_1+c_2+\cdots + c_\nu=1$.
\end{enumerate}
Few remarks are in order. The first item is in fact a consequence of the second, but we stress it here because its meaning is important. The parameter $\nu$ is fixed and decides the whole structure of the matrix-sequence $\{A_n\}_n$. In the second item the equality $c_1+c_2+\cdots c_\nu=1$ is a consequence of the previous limit relations and of the fact that $c_i$ are all positive. Furthermore the positivity of all $c_i$ informs us that all the blocks are essential in determining the spectral and singular value distribution in the Weyl sense; see Remark \ref{rmrk:asumptions 1-2 vs extradim theorems} for more details. We report such notions below in a very general formulation, since various concrete examples need such a generality.

\begin{definition} \cite{GLT-blocks-d-dim,GSI,GSII,MR0890515,TyZ}
\label{def:distributions}
Let ${f}:\Omega \to\mathbb{C}^{s\times t}$ be a  measurable function defined on a measurable set $\Omega\subset\mathbb R^\ell$ with $\ell\ge 1$,
$0<m_\ell(\Omega)<\infty$, $r=\min\{s,t\}$.
Let $\mathcal{C}_0(\mathbb{K})$ be the set of continuous functions with compact support over $\mathbb{K}\in \{\mathbb{C}, \mathbb{R}_0^+\}$ and let $\{A_{{n}}\}_{{n}}$ be a sequence of matrices with singular values $\sigma_j(A_{{n}})$, $j=1,\ldots,n$.
\begin{itemize}
  \item The matrix-sequence  $\{A_{{n}}\}_{n}$ is distributed as ${f}$ in the sense of the singular values, and we write
  \begin{align*}
      \{A_{{n}}\}_{{n}}\sim_\sigma{f},\nonumber
     \end{align*}
    if the following limit relation holds for all $F\in\mathcal{C}_0(\mathbb{R}_0^+)$

  \begin{align}
		  \lim_{{n}\to\infty}\frac{1}{{{n}}}\sum_{j=1}^{{{n}}}F(\sigma_j(A_{n}))=
		  \frac1{m_\ell(\Omega)}\int_{\Omega}  \frac{\sum_{i=1}^{r}F\left(\sigma_i\left({f} \left(\boldsymbol{\theta}\right)\right)\right)}{r}\,\dd{\boldsymbol{\theta}}.\label{eq:distribution_sv}
		 \end{align}

    The function ${f}$ is called the {singular value symbol} which describes the singular value distribution of the matrix-sequence $ \{A_{{n}}\}_{{n}}$.
	\item In the case where $s=t$ and any $A_n$ is square of size $n$ with eigenvalues $\lambda_j(A_{{n}})$, $j=1,\ldots,n$, the matrix-sequence  $\{A_{{n}}\}_{n}$ is distributed as ${f}$ in the sense of the eigenvalues, and we write
  \begin{align*}
      \{A_{{n}}\}_{{n}}\sim_\lambda{f},\nonumber
     \end{align*}
    if the following limit relation holds for all $F\in\mathcal{C}_0(\mathbb{C})$

  \begin{align}
		  \lim_{{n}\to\infty}\frac{1}{{{n}}}\sum_{j=1}^{{{n}}}F(\lambda_j(A_{n}))=
		  \frac1{m_\ell(\Omega)}\int_{\Omega}  \frac{\sum_{i=1}^{s}F\left(\lambda_i\left({f} \left(\boldsymbol{\theta}\right)\right)\right)}{s}\,\dd{\boldsymbol{\theta}}.\label{eq:distribution_eig}
		 \end{align}

    The function ${f} $ is called the eigenvalue symbol and it describes the eigenvalue distribution of the matrix-sequence
    $ \{A_{{n}}\}_{{n}}$.
  \end{itemize}
\end{definition}

\begin{remark}
\label{rmrk:multidimensional}
As it will happen in concrete situations, e.g. for the approximation of multidimensional problems, often it is more convenient to use a more general notation. In that case the matrix-size $n$ is replaced by a quantity ${d_{n}}$ with ${d_{k}}<d_m$ if $k<m$ and/or by the use of a multi-index notation with $n$ replaced by the multi-index $\mathbf{n}$. The multidimensional/multilevel framework is treated in \cite{blocking-gen}.
\end{remark}

In the general framework outlined by Definition   \ref{def:distributions}, the following mathematical tools can be exploited. All of them are related to the idea of the extradimensional approach, indicated in a private conversation  between Tyrtyshnikov and the third author more than two decades ago, and which represents a fruitful framework (see \cite{MR3904142}, beginning of Section 4.2, and \cite{pre-prequel}, Section 2).

\begin{theorem}{\rm \cite{MR3904142}[Theorem 4.3]}\label{th:extradim eigs}
Let $ \{X_{{n}}\}_{{n}}$ be a given Hermitian matrix-sequence with $X_n$ of order $n$. Let $P_n\in \mathbb{C}^{n\times n'}$ be a compression matrix such that $n'<n$,
$P_n^*P_n=I_{n'}$, and let us consider $Y_{n'}=P_n^*X_nP_n$. Under the assumption that
\[
\lim_{{n,n'}\to\infty}\frac{n'}{{{n}}}=1,
\]
i.e. $n=n'+o(n)$, we have
\[
 \{X_{{n}}\}_{{n}}\sim_\lambda{f}\ \ \ \ {\rm iff} \ \ \ \ \{Y_{{n'}}\}_{{n'}}\sim_\lambda{f}.
\]
\end{theorem}

\ \\
\noindent

The following result is the singular value counterpart of Theorem  \ref{th:extradim eigs} and its proof relies on it. A specific singular value version for the square case, with an alternative proof, can be found in \cite{MR3904142}[Corollary 4.4].

\begin{theorem}\label{th:extradim sv}
Let $ \{X_{{n}}\}_{{n}}$ be a given (rectangular) matrix-sequence with $X_n$ of order $n^{(1)}\times n$, $n^{(1)}\ge n$. Let $P_n\in \mathbb{C}^{n\times n'}$, $P_{n^{(1)}}\in \mathbb{C}^{n^{(1)}\times n^{(1),'}}$ be two compression matrices with $n'<n$, $n^{(1),'}< n^{(1)}$ and
$P_n^*P_n=I_{n'}$, $P_{n^{(1)}}^* P_{n^{(1)}}=I_{n^{(1),'}}$, and let us consider $Y_{n'}=P_{n^{(1)}}^*X_nP_n$. Under the assumption that
\[
\lim_{{n,n'}\to\infty}\frac{n'}{{{n}}}=\lim_{{n^{(1)},n^{(1),'}}\to\infty}\frac{n^{(1),'}}{{{n^{(1)}}}}=1,
\]
i.e. $n=n'+o(n)$, $n^{(1)}=n^{(1),'}+ o(n^{(1)})$, we have
\[
 \{X_{{n}}\}_{{n}}\sim_\sigma{f}\ \ \ \ {\rm iff} \ \ \ \ \{Y_{{n'}}\}_{{n'}}\sim_\sigma{f}.
\]
\end{theorem}
\begin{proof}
The singular values squared of $X_{{n}}$ correspond to the eigenvalues of $X_{{n}}^*X_{{n}}$ and the singular values  squared of $Y_{{n'}}$ are the eigenvalues of $Y_{{n'}}^*Y_{{n'}}$. As a consequence, exploiting Theorem \ref{th:extradim eigs} and taking into account Definition \ref{def:distributions} with $F(\sqrt{\cdot})$ in (\ref{eq:distribution_eig}) in place of $F(\cdot)$, we  derive  (\ref{eq:distribution_sv}) for any arbitrary test function $F$. This completes the proof of the desired result.
\end{proof}

We highlight that the assumption
\[
\lim_{{n,n'}\to\infty}\frac{n'}{{{n}}}=1
\]
in Theorem \ref{th:extradim eigs} and the assumption
\[
\lim_{{n,n'}\to\infty}\frac{n'}{{{n}}}=\lim_{{n^{(1)},n^{(1),'}}\to\infty}\frac{n^{(1),'}}{{{n^{(1)}}}}=1
\]
in Theorem \ref{th:extradim sv} are both crucial, and it's straightforward to construct counterexamples for any kind of violation.
For example we consider $X_n={\rm diag}_{i=1,\ldots,n}a\left(\frac{i}{n}\right)$ with $a$ real-valued Riemann integrable over $[0,1]$, choosing $n'$ such that
\[
\lim_{{n,n'}\to\infty}\frac{n'}{{{n}}}=\frac{1}{2},
\]
so violating the above assumptions of Theorem \ref{th:extradim eigs} and of
 Theorem \ref{th:extradim sv}. We define $P_n$ as
 \[
 P_n=\left[\begin{array}{c}
I_{n'} \\
\hline
O_{(n-n')\times n'}
\end{array} \right].
 \]
We have $P_n^*P_n=I_{n'}$ as required and $\{X_{{n}}\}_{{n}}\sim_{\lambda,\sigma} {a}$, but $\{Y_{{n'}}\}_{{n'}}\sim_{\lambda,\sigma} {\hat a}$ with $Y_{{n'}}=P_n^* X_n P_n$, $\hat a (x)=a\left(\frac{x}{2}\right)$, $x\in [0,1]$ and clearly $\hat a \neq a$ unless $a$ is constant.
In fact, in the case where ${f}$ is essentially a scalar, the requirement on the dimensions can be substantially weakened by maintaining one of the implications as proven in \cite{pre-prequel}.

\begin{theorem}\label{th:extradim eigs-f const}
Let $ \{X_{{n}}\}_{{n}}$ be a given Hermitian matrix-sequence with $X_n$ of order $n$. Let $P_n\in \mathbb{C}^{n\times n'}$ be a compression matrix such that $n'<n$, $P_n^*P_n=I_{n'}$, and let us consider $Y_{n'}=P_n^*X_nP_n$. Under the assumption that
\[
\liminf_{{n,n'}\to\infty}\frac{n'}{{{n}}}>0,
\]
we have
\[
 \{X_{{n}}\}_{{n}}\sim_\lambda \alpha I_s\ \ \ \ {\rm implies} \ \ \ \ \{Y_{{n'}}\}_{{n'}}\sim_\lambda \alpha I_s,
\]
with $\alpha$ real constant.
\end{theorem}

The following is the singular value version of Theorem \ref{th:extradim eigs-f const} and it heavily relies on it; see again \cite{pre-prequel}.

\begin{theorem}\label{th:extradim sv-f const}
Let $ \{X_{{n}}\}_{{n}}$ be a given (rectangular) matrix-sequence with $X_n$ of order $n^{(1)}\times n$, $n^{(1)}\ge n$. Let $P_n\in \mathbb{C}^{n\times n'}$, $P_{n^{(1)}}\in \mathbb{C}^{n^{(1)}\times n^{(1),'}}$ be two compression matrices with $n'<n$, $n^{(1),'}< n^{(1)}$ and
$P_n^*P_n=I_{n'}$, $P_{n^{(1)}}^* P_{n^{(1)}}=I_{n^{(1),'}}$, and let us consider $Y_{n'}=P_{n^{(1)}}^*X_nP_n$. Under the assumption that
\[
\liminf_{{n,n'}\to\infty}\frac{n'}{{{n}}}>0, \ \liminf_{{n^{(1)},n^{(1),'}}\to\infty}\frac{n^{(1),'}}{{{n^{(1)}}}}>0,
\]
we have
\[
 \{X_{{n}}\}_{{n}}\sim_\sigma \alpha I_s\ \ \ \ {\rm implies} \ \ \ \ \{Y_{{n'}}\}_{{n'}}\sim_\sigma \alpha I_s,
\]
with $\alpha=0$, while for $\alpha$ real positive constant we need to make the second assumption stronger i.e.
$\lim_{{n^{(1)},n^{(1),'}}\to\infty}\frac{n^{(1),'}}{{{n^{(1)}}}}=1$.
\end{theorem}

\begin{remark}\label{rmrk:asumptions 1-2 vs extradim theorems}
Theorem \ref{th:extradim eigs} and Theorem \ref{th:extradim sv} clarify the importance of assumption 1. and 2. at the beginning of the section.
If there exists an index $j\in \{1,\ldots,\nu\}$ such that $c_j=0$, then this is equivalent to say that $n_j=o(n)$  and hence all the blocks
$A_{j,k}, A_{k,j}$, $k=1,\ldots,\nu$, represent globally a $o(n)$-rank perturbation that can be neglected as stated in the extradimensional Theorem \ref{th:extradim eigs} and Theorem \ref{th:extradim sv}, when deducing the global eigenvalue or singular value distribution of the matrix-sequence $\{A_{{n}}\}_{{n}}$.
\end{remark}

\section{Structured matrices with block unilevel Toeplitz blocks}
\label{sec:Toeplitz_block}
While in Section \ref{sec:general_block} we introduce structured matrices with general blocks, in the following we focus  on matrices of the form (\ref{eq:A_general}) with the additional requirement that each block $A_{i,j}$ possesses a block unilevel Toeplitz structure. Specifically, we examine cases where these Toeplitz matrices are linked to a univariate matrix-valued function $f$, called \textit{generating function}.
Refer to \cite{BS,MR0890515} for the notion in the scalar-valued case and \cite{GLT-blocks-1-dim} for the matrix-valued setting.

\begin{definition}
\label{def:toeplitz}
Let $f$ be a $s\times t$ matrix-valued function belonging to $L^1([-\pi,\pi],s\times t)$ and periodically extended to the whole real line. A Toeplitz matrix $T_{n}(f) \in \mathbb{C}^{sn \times tn}$ associated with $f$ is an $sn\times tn$ matrix defined as
  \begin{align*}
  T_n(f)=\left[\hat f_{i-j}\right]_{i,j=1}^n,\nonumber
 \end{align*}
where
  \begin{align*}
  \hat{f}_{k}=\frac1{2\pi}\int_{-\pi}^{\pi}\!\!f(\theta)\,\E^{-k\iota \theta}\dd\theta \in \mathbb{C}^{s \times t},\qquad k\in\mathbb Z,\qquad \iota^2=-1,
 \end{align*}
are the Fourier coefficients of $f$. For $s=t$ the matrix $T_{n}(f)$ is square.
\end{definition}

\begin{remark}
\label{rmrk:rectangular_toeplitz}
 From Definition \ref{def:toeplitz}, for $n'\neq n$ it is possible to define a (inherently) rectangular $sn\times tn'$  matrix $T_{n,n'}(f)$ as
\begin{equation}
\label{eq:rectangular_Toeplitz}
T_{n,n'}(f)= \left[\hat f_{i-j}\right]_{\stackrel{\substack{i=1,\ldots,n \\ j=1,\ldots,n'}}{}
}.
\end{equation}
\end{remark}
Consequently, we are interested in  matrices of the form (\ref{eq:A_general}) where
\begin{itemize}
\item the diagonal blocks $A_{ii}$ are $sn_i\times tn_i$ Toeplitz matrices associated to the functions $f_{i,i}$, $i=1,\dots, \nu$, respectively;
\item the off-diagonal blocks $A_{ij}$, $i\neq j$, are $sn_{i}\times tn_{j}$ rectangular matrices $T_{n_i,n_j}(f_{i,j})$ according to the definition in (\ref{eq:rectangular_Toeplitz}) of Remark \ref{rmrk:rectangular_toeplitz}.
\end{itemize}
That is, $A_n$ possesses the form
\begin{small}
\begin{equation}
\label{eq:A_Toeplitz}
A_n=\left[
\begin{tikzpicture}[baseline=(m.center)]
\matrix (m) [matrix of math nodes,column sep=0.15em,row sep=0.15em] {
|[draw,dashed, minimum width=1.8cm, minimum height=2cm]| T_{n_1}(f_{1,1}) & |[draw,dashed, minimum width=2.4cm, minimum height=2cm]| T_{n_1,n_2}(f_{1,2}) & \cdots & |[draw,dashed, minimum width=1.6cm, minimum height=2cm]| T_{n_1,n_{\nu}}(f_{1,{\nu}}) \\
|[draw,dashed, minimum width=1.6cm, minimum height=2.4cm]| T_{n_2,n_1}(f_{2,1}) & |[draw,dashed, minimum width=2.4cm, minimum height=2.4cm]| T_{n_2}(f_{2,2}) & \cdots & |[draw,dashed, minimum width=1.6cm, minimum height=2.4cm]| T_{n_2,n_{\nu}}(f_{2,{\nu}}) \\
\vdots & \vdots & \ddots & \vdots \\
|[draw,dashed, minimum width=1.6cm, minimum height=1.6cm]| T_{n_{\nu},n_1}(f_{{\nu},1}) & |[draw,dashed,minimum width=2.4cm, minimum height=1.6cm]| T_{n_{\nu},n_2}(f_{{\nu},2}) & \cdots & |[draw,dashed, minimum width=1.8cm, minimum height=1.6cm]| T_{n_{\nu}}(f_{\nu,\nu}) \\
};
\end{tikzpicture}
\right],
\end{equation}
\end{small}
where $f_{i,j}$, $i,j=1,\dots,\nu$, are functions of the type described in Definition \ref{def:toeplitz}.

\begin{theorem}[\cite{MR1671591}]\label{thm:szego-d=1}
Let ${f}:[-\pi,\pi]\to\mathbb{C}^{s \times t}$, ${f}\in L^1([-\pi,\pi],s\times t)$.   Then
	\begin{equation}
		\label{eq:szego-sv-d=1}
		\{T_{n}({f})\}_{{n}}\sim_\sigma~{f}.
	\end{equation}
If in addition the generating function is a Hermitian-valued function, then necessarily $s=t$ and 	
	\begin{equation}
		\label{eq:szego-eig-d=1}
		\{T_{n}({f})\}_{{ n}}\sim_\lambda~{f}.
	\end{equation}
\end{theorem}

The existence of distributions analogous to (\ref{eq:distribution_sv}) has been established   for block multilevel Hankel matrices generated by the Fourier coefficients of a matrix-valued multivariate measure whose singular part is finitely supported \cite{MR1740439}. In the context of this work, we present a version where the generating function is univariate and matrix-valued.

\begin{theorem}[\cite{MR1740439}]\label{thm:fasino-tilli-d=1}
Let ${f}:[-\pi,\pi]\to\mathbb{C}^{s \times t}$, { ${f}\in L^1([-\pi,\pi],s\times t)$}. Then
	\begin{equation}
		\label{eq:fasino-tilli-d=1}
		\{H_{n}({f})\}_{{n}}\sim_\sigma~0,
	\end{equation}
	where $H_{n}({f})$ is the Hankel matrix given by
	\begin{equation}
	\label{eq:hankel_def-d=1}
	H_{n}({f})= \sum_{k=1}^{2n-1} K_{n}^{(k)} \otimes  \hat{{f}}_{k},
	\end{equation}
	and  $K_{n}^{(k)}$ denotes the matrix of order $n$ whose $(i, j )$ entry equals
1 if $j + i = k + 1$ and equals zero otherwise, while $\hat{{f}}_{k}$, $k \in\mathbb Z$,  are the Fourier coefficients of $f$ as given in Definition \ref{def:toeplitz}.
\end{theorem}

\begin{theorem}\label{equal-blocks-basic}
Assume that a structure $A_n$ as in (\ref{eq:A_Toeplitz}) is given with $n_1=n_2=\cdots=n_\nu=n(\nu)=\frac{n}{\nu}$. For any positive integer $\mu$ define the following permutation matrix $\Pi_\mu$
\[
\Pi_\mu=\Pi_{n,\nu,\mu}
= \begin{bmatrix}
I_\nu \otimes \mathbf{e}_{1}^T \\
I_\nu \otimes \mathbf{e}_{2}^T \\
\vdots \\
I_\nu \otimes \mathbf{e}_{n(\nu)}^T
\end{bmatrix} \otimes I_\mu
\]
is such that
\[
T_n(F)=\Pi_s A_n \Pi_t^T
\]
with $F=\left(f_{i,j}\right)_{i,j=1}^\nu$. Hence
\[
\{A_{n}\}_{{n}}\sim_\sigma~{F}
\]
and
\[
\{A_{n}\}_{{n}}\sim_\lambda~{F}
\]
if in addition $F$ is Hermitian-valued (which implies necessarily $s=t$).
\end{theorem}
\begin{proof}
The permutation matrix $\Pi_\mu=\Pi_{n,\nu,\mu}$ is given in \cite{GLT-blocks-1-dim}[Equation (2.1), p. 38] and for $s=t$ the statement $T_n(F)=\Pi_s A_n \Pi_t^T$ is contained in  \cite{GLT-blocks-1-dim}[Theorem 2.42, p. 56]: the generalization to $s\neq t$ is plain and so in general $T_n(F)=\Pi_s A_n \Pi_t^T$. Since permutation matrices are unitary, by the first part of Theorem \ref{thm:szego-d=1}, we deduce $\{A_{n}\}_{{n}}\sim_\sigma~{F}$. Furthermore, when $F$ is Hermitian-valued we necessarily have $s=t$. Therefore $T_n(F)=\Pi_s A_n \Pi_s^T$ and $A_n$ are similar and by the second part of Theorem \ref{thm:szego-d=1} we deduce $\{A_{n}\}_{{n}}\sim_\lambda~{F}$.
\end{proof}

\begin{theorem}\label{almost equal-blocks-basic}
Assume that a structure $A_n$ as in (\ref{eq:A_Toeplitz}) is given with $n_j=n(\nu)+o(n)$, $j=1,\ldots,\nu$, $n(\nu)=\frac{n}{\nu}$, and with $F=\left(f_{i,j}\right)_{i,j=1}^\nu$.
Then
\[
\{A_{n}\}_{{n}}\sim_\sigma~{F}
\]
and
\[
\{A_{n}\}_{{n}}\sim_\lambda~{F}
\]
if in addition $F$ is Hermitian-valued.
\end{theorem}
\begin{proof}
The proof amounts in combining Theorem \ref{equal-blocks-basic} and Theorems \ref{th:extradim eigs}-\ref{th:extradim sv}.
\end{proof}
\ \\
\noindent

Now we are ready for an illustrative example, which clarify how to use the theoretical results in the case of $\nu=2$ i.e.
\begin{equation}\label{eq:nu=2}
A_n=
\begin{bmatrix}
T_{n_1}(f_{1,1}) &  T_{n_1,n_2}(f_{1,2})\\
T_{n_2,n_1}(f_{2,1}) &  T_{n_2}(f_{2,2})
\end{bmatrix}
\end{equation}
with $s=t=1$.
The case reported in Theorem \ref{equal-blocks-basic} with $\nu=2$ and $s=t=1$ comes from the recovery of missing data (see \cite{missing-data2}[Section 4.2]) and already Theorem \ref{almost equal-blocks-basic} allows to extend to validity of the analysis in \cite{missing-data2}. Further extensions can be obtained.
We assume that $n_1$, $n_2$ satisfy the assumptions at the beginning of Section \ref{sec:general_block}. Moreover, we require that $c_1\in \mathbb Q^+$ (so that necessarily $c_2=1-c_1\in  \mathbb Q^+$). Hence
\[
c_1=\frac{\alpha_1}{\beta_1}, \ \ \ c_2=\frac{\alpha_2}{\beta_2}, \ \ \ \alpha_1,\alpha_2,\beta_1, \beta_2 \in \mathbb N^+
\]
so that
\begin{equation}\label{reduction to mcm}
c_1=\frac{m_1}{m}, \ \ \ c_2=\frac{m_2}{m}, \ \ \ \ m={\rm lcm}(\beta_1,\beta_2).
\end{equation}
Starting from the structure in (\ref{eq:nu=2}) and considering the terms in (\ref{reduction to mcm}), we define the matrix-valued function $F$ which depends on $f_{j,k}$, $j,k=1,2$, and on $m_1,m_2,m$ as follows:
$E_{1,1}=f_{1,1}I_{m_1}$, $E_{2,2}=f_{2,2}I_{m_2}$ and for $j\neq k$, $j,k=1,2$, we have
\[
E_{j,k} =
\begin{cases}
\begin{array}{l l l l}
\left[
\begin{array}{c|c}
f_{j,k} I_{m_j} & O_{m_j\times (m_k-m_j)} \\
\end{array}
\right] &  &\text{if } j\neq k, & \text{and } m_k\ge m_j, \\
 & & \\
\left[\begin{array}{c}
f_{j,k}I_{m_k} \\
\hline
O_{(m_j-m_k)\times m_k}
\end{array} \right]& & \text{if } j\neq k, & \text{and } m_j\ge m_k. \\
\end{array}
\end{cases}
\]
The idea is that the singular value symbol is determined only by $T_{n_j}(f_{j})$ and by the largest principal submatrix of $T_{n_j,n_k}(f_{j,k})$ i.e.
$T_{\min\{n_j,n_k\}}(f_{j,k})$, $j\neq k$, $j,k=1,2$, the remaining part being a flip of a Hankel matrix, whose associated matrix-sequence is zero-distributed according to Theorem \ref{thm:fasino-tilli-d=1}.
Under the above assumptions and the above notation we conclude that
\begin{equation}\label{distr-sv:nu=2-mcm}
\{A_{n}\}_{{n}}\sim_\sigma~{F}=
\begin{bmatrix}
E_{1,1} &  E_{1,2}\\
E_{2,1} &  E_{2,2}
\end{bmatrix}
\end{equation}
and
\begin{equation}\label{distr-eig:nu=2-mcm}
\{A_{n}\}_{{n}}\sim_\lambda~{F}=
\begin{bmatrix}
E_{1,1} &  E_{1,2}\\
E_{2,1} &  E_{2,2}
\end{bmatrix}
\end{equation}
if in addition $A_n$ is Hermitian for any $n_1,n_2$.

In fact we are ready for the generalization of the previous qualitative steps to the case of $\nu\ge 2$, accompanied by its precise derivations.

\begin{lemma}\label{lemma1:from fasino-tilli-d=1}
Let ${f}:[-\pi,\pi]\to\mathbb{C}^{s \times t}$, { ${f}\in L^1([-\pi,\pi],s\times t)$}, and let us consider $T_{n,n'}(f)$ with
\begin{equation}\label{nondegenerate}
0< \liminf_{n,n'\rightarrow \infty}\frac{n}{n+n'} \le \limsup_{n,n'\rightarrow \infty}\frac{n}{n+n'} <1,
	\end{equation}
that is $\max\{n,n'\}-\min\{n,n'\}\sim n+n'$ (far away from the square case when $s=t$) and $\min\{n,n'\}\sim n+n'$ (far away from degenerate rectangular case).
 If $n>n'$ then
	\begin{equation}
		\label{eq1: from fasino-tilli-d=1}
T_{n,n'}(f)=\left[\begin{array}{c}
T_{n'}(f) \\
\hline
X_{n,n'}
\end{array} \right], \ \ \
		\{X_{n,n'}\}_{{n}}\sim_\sigma~0.
	\end{equation}
 If $n'>n$ then
	\begin{equation}
	\label{eq1: from fasino-tilli-d=1-bis}
	T_{n,n'}(f)=\left[
\begin{array}{c|c}
T_{n}(f) & Y_{n,n'} \\
\end{array}
\right], \ \ \
		\{Y_{n,n'}\}_{{n}}\sim_\sigma~0.
	\end{equation}
\end{lemma}
\begin{proof}
Let us look closely to the matrix $X_{n,n'}$ as in the first expression in (\ref{eq1: from fasino-tilli-d=1}). From direct inspection we that the Fourier coefficient $\hat{{f}}_{1}$ is present only one time in the matrix (in the north-east corner of $X_{n,n'}$), the Fourier coefficient $\hat{{f}}_{2}$ is present only twice in the matrix (in the north-east corner of $X_{n,n'}$), and so on. Hence if we make a block flip from the right, then we obtain that $X_{n,n'} (W_{n'}\otimes I_t)$ is submatrix of size $s(n-n')\times tn'$ of $H_{\tilde n}(f)$ with $\tilde n=\max\{n-n',n'\}$ and with $W_\mu$ being the anti-identity (or flip) matrix-size of size $\mu$. Since $\{H_{\tilde n}\}_{{n}}\sim_\sigma~0$ by Theorem \ref{thm:fasino-tilli-d=1}, taking into account that $W_{n'}\otimes I_t$ is unitary so that $X_{n,n'} (W_{n'}\otimes I_t)$ and $X_{n,n'}$ share the same singular values, the application of Theorem \ref{th:extradim sv-f const} with $\alpha=0$ granted by the assumptions in (\ref{nondegenerate}) implies that
\[
\{X_{n,n'}\}_{{n}}\sim_\sigma~0.
\]
If $n'>n$ and we look closely to the matrix $Y_{n,n'}$ as in the first expression in (\ref{eq1: from fasino-tilli-d=1-bis}), then there is technical difficulty. Indeed the Fourier coefficient $\hat{{f}}_{-1}$ is present only one time in the matrix (in the south-west corner of $Y_{n,n'}$), the Fourier coefficient $\hat{{f}}_{-2}$ is present only twice in the matrix (in the south-west corner of $Y_{n,n'}$), and so on. However the negative Fourier coefficients of $f(\cdot)$ are the positive Fourier coefficients of $f(-\cdot)$ that is $\hat{{f^R}}_{k}=\hat{{f}}_{-k}$ for every $k\in \mathbb Z$, with $f^R(\theta)=f(-\theta)$ and with $f^R \in L^1([-\pi,\pi],s\times t)$ owing to $f\in L^1([-\pi,\pi],s\times t)$.
Now the rest of the proof mimics that concerning the previous case: by making block flipping from the left, we obtain that $(W_n\otimes I_s) Y_{n,n'}$ is of Hankel nature and is a submatrix of size $sn \times t(n'-n)$ of $H_{\tilde n}(f)$ with $\tilde n=\max\{n'-n,n\}$. Since $\{H_{\tilde n}\}_{{n}}\sim_\sigma~0$ by Theorem \ref{thm:fasino-tilli-d=1}, taking into account that $W_{n}\otimes I_s$ is unitary so that $(W_n\otimes I_s) Y_{n,n'}$ and $Y_{n,n'}$ share the same singular values, the application of Theorem \ref{th:extradim sv-f const} with $\alpha=0$ granted by the assumptions in (\ref{nondegenerate}) implies that
\[
\{Y_{n,n'}\}_{{n}}\sim_\sigma~0.
\]
\end{proof}
\ \\
\noindent

The previous lemma is quite simple in itself, but it has surprisingly strong consequences. In order to describe such consequences let us introduce the following auxiliary matrices $\tilde A_n$, $\hat A_n$, in which with respect to $A_n$ in (\ref{eq:A_Toeplitz}) terms leading to zero-distributed matrix-sequences are simply neglected.
More precisely we set
\begin{small}
\begin{equation}
\label{eq:tilde A_Toeplitz}
\tilde A_n=\left[
\begin{tikzpicture}[baseline=(m.center)]
\matrix (m) [matrix of math nodes,column sep=0.15em,row sep=0.15em] {
|[draw,dashed, minimum width=2cm, minimum height=2cm]| T_{n_1}(f_{1,1}) & |[draw,dashed, minimum width=2.4cm, minimum height=2cm]| \tilde T_{n_1,n_2} & \cdots & |[draw,dashed, minimum width=1.6cm, minimum height=2cm]| \tilde T_{n_1,n_{\nu}} \\
|[draw,dashed, minimum width=2cm, minimum height=2.4cm]| \tilde T_{n_2,n_1} & |[draw,dashed, minimum width=2.4cm, minimum height=2.4cm]| T_{n_2}(f_{2,2}) & \cdots & |[draw,dashed, minimum width=1.6cm, minimum height=2.4cm]| \tilde T_{n_2,n_{\nu}} \\
\vdots & \vdots & \ddots & \vdots \\
|[draw,dashed, minimum width=2cm, minimum height=1.6cm]| \tilde T_{n_{\nu},n_1} & |[draw,dashed,minimum width=2.4cm, minimum height=1.6cm]| \tilde T_{n_{\nu},n_2} & \cdots & |[draw,dashed, minimum width=1.6cm, minimum height=1.6cm]| T_{n_{\nu}}(f_{\nu,\nu}) \\
};
\end{tikzpicture}
\right],
\end{equation}
\end{small}
with
\[
\tilde T_{n,n'}=\left[\begin{array}{c}
T_{n'}(f) \\
\hline
O_{s(n-n')\times tn'}
\end{array} \right],
\]
if $n\ge n'$ and
\[
\tilde T_{n,n'}=\left[
\begin{array}{c|c}
T_{n}(f) & O_{sn\times t(n'-n)} \\
\end{array}
\right],
\]
if $n'>n$.

Now we define
\begin{small}
\begin{equation}
\label{eq:hat A_Toeplitz}
\hat A_n=\left[
\begin{tikzpicture}[baseline=(m.center)]
\matrix (m) [matrix of math nodes,column sep=0.15em,row sep=0.15em] {
|[draw,dashed, minimum width=2cm, minimum height=2cm]| \hat T_{n_1} & |[draw,dashed, minimum width=2.4cm, minimum height=2cm]| \hat T_{n_1,n_2} & \cdots & |[draw,dashed, minimum width=1.6cm, minimum height=2cm]| \hat T_{n_1,n_{\nu}} \\
|[draw,dashed, minimum width=2cm, minimum height=2.4cm]| \hat T_{n_2,n_1} & |[draw,dashed, minimum width=2.4cm, minimum height=2.4cm]| \hat T_{n_2} & \cdots & |[draw,dashed, minimum width=1.6cm, minimum height=2.4cm]| \hat T_{n_2,n_{\nu}} \\
\vdots & \vdots & \ddots & \vdots \\
|[draw,dashed, minimum width=2cm, minimum height=1.6cm]| \hat T_{n_{\nu},n_1} & |[draw,dashed,minimum width=2.4cm, minimum height=1.6cm]| \hat T_{n_{\nu},n_2} & \cdots & |[draw,dashed, minimum width=1.6cm, minimum height=1.6cm]| \hat T_{n_{\nu}} \\
};
\end{tikzpicture}
\right],
\end{equation}
\end{small}
where we assume that the sizes $n_1,\ldots,n_\nu$ can be written as $n_j=m n_j'+o(n_j)=m_j n(m)+o(n)$, $m$ fixed positive integer, and $n(m)=\frac{n}{m}$.
Under such a setting, by considering the block $(j,k)$ in (\ref{eq:tilde A_Toeplitz}) with $n_j\le n_k$, the matrix  $T_{n_{j}}(f)$ in the considered block is replaced essentially by  $m_j$ copies of $T_{n(m)}(f)$ i.e. $\left[I_{m_j-1}\otimes T_{n(m)}(f)\right]\oplus X_{n(m),j,k}$ and $X_{n(m),j,k}$ is a matrix of size $n(m)+o(n(m))$, obtained by neglecting, if the sign of the $o(n(m))$ term is negative, or adding, if the sign of the $o(n(m))$ term is positive, $o(n(m))$ last rows and columns in $T_{n(m)}(f)$. Similarly, by considering the block $(j,k)$ in (\ref{eq:tilde A_Toeplitz}) with $n_k\le n_j$, the matrix  $T_{n_{k}}(f)$ in the considered block is replaced essentially by  $m_k$ copies of $T_{n(m)}(f)$ i.e. $\left[I_{m_k-1}\otimes T_{n(m)}(f)\right]\oplus X_{n(m),j,k}$ and $X_{n(m),j,k}$ is a matrix of size $n(m)+o(n(m))$, obtained by neglecting, if the sign of the $o(n(m))$ term is negative, or adding, if the sign of the $o(n(m))$ term is positive, $o(n(m))$ last rows and columns in $T_{n(m)}(f)$.

\begin{proposition}\label{prop:equivalence 1}
Assume that a structure $A_n$ as in (\ref{eq:A_Toeplitz}) is given with $n_j$, $j=1,\ldots,\nu$, satisfying conditions 1. and 2. at the beginning of Section \ref{sec:general_block}. Let $\tilde A_n$ defined as in (\ref{eq:tilde A_Toeplitz}). Then
	\begin{equation}
		\label{eq:equiv1}
		\{A_{n}\}_{{n}}\sim_\sigma G \ \ {\rm iff}\ \ \{\tilde A_{n}\}_{{n}}\sim_\sigma G.
	\end{equation}
If in addition $A_n$ and $\tilde A_n$ are Hermitian then
	\begin{equation}
		\label{eq:equiv2}
		\{A_{n}\}_{{n}}\sim_\lambda G \ \ {\rm iff}\ \ \{\tilde A_{n}\}_{{n}}\sim_\lambda G.
	\end{equation}
\end{proposition}
\begin{proof}
In \cite{GLT-blocks-1-dim}[Remark 2.35] it is stated that $\{X_{n}-Y_n\}_{{n}}\sim_\sigma 0$ implies
\[
\{X_{n}\}_{{n}}\sim_\sigma G \ \ {\rm iff}\ \ \{Y_{n}\}_{{n}}\sim_\sigma G
\]
with $\{X_{n}\}_{{n}}\sim_\lambda G$ iff $\{Y_{n}\}_{{n}}\sim_\lambda G$ if in addition the matrices $X_n,Y_n$ are Hermitian for every $n$. Hence the proof amounts in proving that $\{A_{n}-\tilde A_n\}_{{n}}\sim_\sigma 0$. Now looking at the construction in (\ref{eq:tilde A_Toeplitz}) and taking into account Lemma \ref{lemma1:from fasino-tilli-d=1}, we obtain that the matrix-sequence $\{A_{n}-\tilde A_n\}_{{n}}$ can be written as the sum of $\nu^2-\nu$ zero-distributed matrix-sequences in which for any zero-distributed summand the only nonzero submatrix is of the form either (\ref{eq1: from fasino-tilli-d=1}) or (\ref{eq1: from fasino-tilli-d=1-bis}). We stress that the assumption (\ref{nondegenerate}) in Lemma \ref{lemma1:from fasino-tilli-d=1} is satisfied thanks to conditions 1. and 2. reported at the beginning of Section \ref{sec:general_block}.

Since $\nu$ is a constant independent of $n$ by condition 1., we infer that also $\{A_{n}-\tilde A_n\}_{{n}}$ is zero-distributed and this concludes the proof.
\end{proof}
\ \\
\noindent

\begin{proposition}\label{prop:equivalence 2}
Assume that a structure $A_n$ as in (\ref{eq:A_Toeplitz}) is given with $n_j$, $j=1,\ldots,\nu$, satisfying conditions 1. and 2. at the beginning of Section \ref{sec:general_block}. Let $\hat A_n$ defined as in (\ref{eq:hat A_Toeplitz}). Then
	\begin{equation}
		\label{eq:equiv1-bis}
		\{A_{n}\}_{{n}}\sim_\sigma G \ \ {\rm iff}\ \ \{\hat A_{n}\}_{{n}}\sim_\sigma G.
	\end{equation}
If in addition $A_n$ and $\tilde A_n$ are Hermitian then
	\begin{equation}
		\label{eq:equiv2-bis}
		\{A_{n}\}_{{n}}\sim_\lambda G \ \ {\rm iff}\ \ \{\hat A_{n}\}_{{n}}\sim_\lambda G.
	\end{equation}
\end{proposition}
\begin{proof}

As in the proof of Proposition \ref{prop:equivalence 1} the key fact is stated in \cite{GLT-blocks-1-dim}[Remark 2.35]. Furthermore in the proof of Proposition \ref{prop:equivalence 1} we have shown that $\{A_{n}-\tilde A_n\}_{{n}}$ can be written as the sum of $\nu^2-\nu$ zero-distributed matrix-sequences. Since $\{A_{n}-\hat A_n\}_{{n}}=\{A_{n}-\tilde A_n\}_{{n}}+\{\tilde A_{n}-\hat A_n\}_{{n}}$ and $\nu$ is fixed with respect to $n$ by assumption 1., it is now sufficient to show that also $\{\tilde A_{n}-\hat A_n\}_{{n}}$ can be written as the sum of a constant number of zero-distributed matrix-sequences. In order to do that we carefully look at the construction of $\hat A_n$ in (\ref{eq:hat A_Toeplitz}). Since the only change with respect to $\tilde A_n$ is made on the diagonal blocks, we deduce that the structure of $\tilde A_{n}-\hat A_n$ is the same as in (\ref{eq:hat A_Toeplitz}), but in each block only the nondiagonal part survives, since in $\hat A_n$ we copy the diagonal part of each block of $\tilde A_{n}$.

Taking $n_{*}$ equal to either $n_{j}$ or $n_{k}$, according to the mathematical description after equation (\ref{eq:hat A_Toeplitz}), a careful direct check shows that each block
\[
T_{n_{*}}(f_{j,k})-\left(\left[I_{m_*-1} \otimes T_{n(m)}(f_{j,k})\right]\oplus X_{n(m),j,k}\right)
\]
can be expressed as the sum of $m_*^2-m_*$ blocks, which are flippings of submatrices either of $H_{n_{*}}(f_{j,k})$ or $H_{n_{*}}(f_{j,k}^R)$ with $f^R_{j,k}(\theta)=f_{j,k}(-\theta)$, including terms of the form (\ref{eq1: from fasino-tilli-d=1}) and (\ref{eq1: from fasino-tilli-d=1-bis}) in the main block subdiagonal and in the main block upperdiagonal, respectively. Since $\{H_{n_{*}}(f_{j,k})\}_{{n_j}}\sim_\sigma 0$ and $\{H_{n_{*}}(f_{j,k}^R)\}_{{n_j}}\sim_\sigma 0$ by Theorem \ref{thm:fasino-tilli-d=1}, taking into account that $\nu,m_1,\ldots,m_\nu$ and $m$ are fixed with respect to $n$, by following a reasoning as in Lemma \ref{lemma1:from fasino-tilli-d=1} where the block flippings are explicitly defined, we finally deduce that  $\{\tilde A_{n}-\hat A_n\}_{{n}}$ can be expressed as the sum of $g(\nu,m_1,\ldots,m_\nu,m)$ zero-distributed matrix-sequences, with $g(\nu,m_1,\ldots,m_\nu,m)$ independent of $n$. 

With the latter the desired result is proven.
\end{proof}
\ \\
\noindent

\begin{theorem}\label{th:general}
Assume that a structure $A_n$ as in (\ref{eq:A_Toeplitz}) is given with $n_j$, $j=1,\ldots,\nu$, satisfying conditions 1. and 2. at the beginning of Section \ref{sec:general_block}.
Assume that $c_1,\ldots,c_\nu \in \mathbb Q^+$, $c_j=\frac{\alpha_j}{\beta_j}$, $\alpha_j,\beta_j \in \mathbb N^+$, $j=1,\ldots,\nu$. Then, setting
$F=\left(E_{i,j}\right)_{i,j=1}^\nu$, we have
\[
\{A_{n}\}_{{n}}\sim_\sigma~{F}
\]
and
\[
\{A_{n}\}_{{n}}\sim_\lambda~{F}
\]
if in addition $A_n$ is Hermitian. Here
\[
E_{jk} =
\begin{cases}
\begin{array}{l l l l}
I_{m_j}\otimes f_{j,j} & & \text{if } j=k, \\
 & & \\
\left[
\begin{array}{c|c}
 I_{m_j} \otimes f_{j,k}& O_{sm_j\times t(m_k-m_j)} \\
\end{array}
\right] &  &\text{if } j\neq k, & \text{and } m_k\ge m_j, \\
 & & \\
\left[\begin{array}{c}
I_{m_k} \otimes f_{j,k} \\
\hline
O_{s(m_j-m_k)\times tm_k}
\end{array} \right]& & \text{if } j\neq k, & \text{and } m_j\ge m_k, \\
\end{array}
\end{cases}
\]
with $c_j=\frac{m_j}{m}$,  $j,k=1,\ldots,\nu$, $m={\rm lcm}(\beta_1,\ldots,\beta_\nu)$.
\end{theorem}
\begin{proof}
The core point of the proof consists in a careful look at the matrix $\hat A_n$.

As a first step we assume that $n_j=c_j n$ exactly for any $j=1,\ldots,\nu$. Under the given assumptions on the rational character of $c_j=\frac{m_j}{m}$,  $j=1,\ldots,\nu$, $m={\rm lcm}(\beta_1,\ldots,\beta_\nu)$, the diagonal matrices $\left[I_{m-1} \otimes T_{n'_{j}}(f_{j,j})\right]\oplus X_{n'_{j}}$ reduce exactly to
$I_{m} \otimes T_{n'_{j}}(f_{j,j})$, $j=1,\ldots,\nu$. Hence $\hat A_n$ is a block matrix of the form

\begin{small}
\begin{equation}
\label{eq:A_Toeplitz-equal}
\hat A_n=\left[
\begin{tikzpicture}[baseline=(m.center)]
\matrix (m) [matrix of math nodes,column sep=0.15em,row sep=0.15em] {
|[draw,dashed, minimum width=2cm, minimum height=2cm]| T_{n(m)}(E_{1,1}) & |[draw,dashed, minimum width=2.4cm, minimum height=2cm]| T_{n(m)}(E_{1,2}) & \cdots & |[draw,dashed, minimum width=1.6cm, minimum height=2cm]| T_{n(m)}(E_{1,{\nu}}) \\
|[draw,dashed, minimum width=2cm, minimum height=2.4cm]| T_{n(m)}(E_{2,1}) & |[draw,dashed, minimum width=2.4cm, minimum height=2.4cm]| T_{n(m)}(E_{2,2}) & \cdots & |[draw,dashed, minimum width=1.6cm, minimum height=2.4cm]| T_{n(m)}(E_{2,{\nu}}) \\
\vdots & \vdots & \ddots & \vdots \\
|[draw,dashed, minimum width=2cm, minimum height=1.6cm]| T_{n(m)}(E_{{\nu},1}) & |[draw,dashed,minimum width=2.4cm, minimum height=1.6cm]| T_{n(m)}(E_{{\nu},2}) & \cdots & |[draw,dashed, minimum width=1.6cm, minimum height=1.6cm]| T_{n(m)}(E_{\nu,\nu}) \\
};
\end{tikzpicture}
\right]
\end{equation}
\end{small}
with $n(m)=\frac{n}{m}$ and $E_{j,k}$ as in the assumption for $j,k=1,\ldots,\nu$.
The latter means that we are in the framework of Theorem \ref{equal-blocks-basic}, so that setting $F=\left(E_{i,j}\right)_{i,j=1}^\nu$ we have
\[
\{\hat A_{n}\}_{{n}}\sim_\sigma~{F}
\]
and
\[
\{\hat A_{n}\}_{{n}}\sim_\lambda~{F}
\]
if in addition $F$ is Hermitian-valued (which implies necessarily $s=t$). Now the use of Proposition \ref{prop:equivalence 2} ends the proof in the case where $n_j=c_j n$ exactly for any $j=1,\ldots,\nu$.

The final step concerns the general case where $n_j=c_j n +o(n)$ for any $j=1,\ldots,\nu$. However the claimed thesis is simply obtained from the first step using the powerful extradimensonal approach, i.e. Theorem \ref{th:extradim sv} for the singular value distribution and Theorem \ref{th:extradim eigs} for the eigenvalue distribution.
\end{proof}
\ \\
\noindent

\begin{remark}\label{reducing to essential}
In the proof of Proposition \ref{prop:equivalence 1} and of Proposition \ref{prop:equivalence 2}, the starting point is \cite{GLT-blocks-1-dim}[Remark 2.35] stating that $\{X_{n}-Y_n\}_{{n}}\sim_\sigma 0$ implies $\{X_{n}\}_{{n}}\sim_\sigma G$ iff $\{Y_{n}\}_{{n}}\sim_\sigma G$ with $G$ being a $d$ variate and matrix-valued function. We here add that only needed case is that with $d=s=t=1$ thanks to the non-uniqueness of the singular value symbol (see \cite{nonunique,EMI}) and to its nondecreasing univariate unique rearrangement defined on $[0,1]$ (see \cite{rearr}). Indeed, looking at Definition \ref{def:distributions}, $\{A_{n}\}_{{n}}\sim_\sigma G$ iff $\{A_{n}\}_{{n}}\sim_\sigma \hat G$ where $\hat G$ is any rearrangement of $G$ i.e. ${\hat G}:\hat\Omega \to\mathbb{C}^{s'\times t'}$ is a measurable function defined on a measurable set $\hat \Omega\subset\mathbb R^{\hat \ell}$ with ${\hat \ell}\ge 1$,$0<m_{\hat \ell}(\hat\Omega)<\infty$, $r'=\min\{s',t'\}$ and
for all $F\in\mathcal{C}_0(\mathbb{R}_0^+)$
  \begin{align}
		  		  \frac1{m_\ell(\Omega)}\int_{\Omega}  \frac{\sum_{i=1}^{r}F\left(\sigma_i\left({G} \left(\boldsymbol{\theta}\right)\right)\right)}{r}\,\dd{\boldsymbol{\theta}}=
		  		  \frac1{m_{\hat \ell}(\hat \Omega)}\int_{\hat \Omega}  \frac{\sum_{i=1}^{r'}F\left(\sigma_i\left({\hat G} \left(\boldsymbol{\theta}\right)\right)\right)}{r'}\,\dd{\boldsymbol{\theta}}.
		 \end{align}
 We choose now ${\hat G}=\phi_G$ as the unique nondecreasing rearrangement of $G$ defined on $[0,1]$. Hence $\{X_{n}\}_{{n}}\sim_\sigma G$ iff $\{X_{n}\}_{{n}}\sim_\sigma \phi_G$ and $\{Y_{n}\}_{{n}}\sim_\sigma G$ iff $\{Y_{n}\}_{{n}}\sim_\sigma \phi_G$ from which $\{X_{n}\}_{{n}}\sim_\sigma G$ iff $\{Y_{n}\}_{{n}}\sim_\sigma G$ is equivalent to write $\{X_{n}\}_{{n}}\sim_\sigma \phi_G$ iff $\{Y_{n}\}_{{n}}\sim_\sigma \phi_G$.
\end{remark}
\begin{remark}\label{rmrk:conjecture}
Assume that a class of matrix-sequences $\{\{A_n(l)\}_n: l \}$ is given with $A_n(l)$ as in (\ref{eq:A_Toeplitz}) with $l\in \mathbb{N}^+ $ and with $f_{i,j}$, $i,j=1,\dots,\nu$, fixed functions independent of $l$ as in Definition \ref{def:toeplitz}.
Assume that 1. and 2. are satisfied with the coefficients $c_j(l)$ rational numbers converging to $c_j$, $l\in \mathbb{N}^+$, $j=1,\dots,\nu$. It is clear that the related distribution functions $F_l$ cannot tend to any $F$, simply because according to Theorem \ref{th:general} the size of the symbol
\[
F_l=\left(E_{i,j}(l)\right)_{i,j=1}^\nu
\]
is erratic and depending on the given rational coefficients $c_j(l)$.
However, according to Theorem \ref{th:general}  and Remark \ref{reducing to essential}, we have $\{A_{n}(l)\}_{{n}}\sim_\sigma~{F_l}$ iff $\{A_{n}(l)\}_{{n}}\sim_\sigma~{\phi_{ F_l}}$, with $\phi_{ F_l}$ being the nondecreasing rearrangement of $F_l$ defined on $[0,1]$.
Now it makes sense to ask if $F_l$ converges to some $F$ as $l$ tends to infinity. A positive answer and the use of the notion of approximating class of sequences with the related topology (see \cite{GSI} and references therein)  would represent a possibility of extending Theorem \ref{th:general} to the case where some of the $c_j$ are not rational numbers. A preliminary study with related numerical experiments is presented in \cite{pre-prequel} for the case of $\nu=2$ and $s=t=1$.
Very recently, in \cite{gacs} the notion of approximating class of sequences has been revisited and generalized, including the case where the matrix sizes of the approximants and of the original matrix-sequence are not the same. Such idea relies again on the extradimensional techniques discussed in \cite{pre-prequel} and it has been used in \cite{blocking-conj} for obtaining the distributional results when the $c_j$ are not necessarily rational for $s=t=1$ and $j=1,\dots,\nu$.
\end{remark}

\section{Numerical experiments}\label{sec:num}
The Section contains selected examples which provide the numerical confirmation of the findings in Section \ref{sec:general_block}
and Section \ref{sec:Toeplitz_block}.
The singular value and spectral distributions of a matrix-sequence are defined by equations (\ref{eq:distribution_sv}) and (\ref{eq:distribution_eig}) of Definition \ref{def:distributions}. The following remark provides an informal interpretation of Definition \ref{def:distributions} and outlines the numerical and computational explanation to verify the distributions in practice.

\begin{remark}\label{rmrk:distributions}
Let us consider a setting as in Definition \ref{def:distributions}, taking into account Remark \ref{rmrk:multidimensional}.
Let $\lambda_1(F), \ldots,\lambda_s(F)$ and  $\sigma_1(F),
\ldots,\sigma_s(F)$ be the eigenvalue and the singular value functions of a
$s\times s$ matrix-valued function $F$, respectively. If $F$ is
smooth enough, an informal interpretation of the limit relation
\eqref{eq:distribution_sv} (resp. \eqref{eq:distribution_eig})
is that when the matrix-size of $A_{n}$ is sufficiently large,
then $d_n/s$ eigenvalues (resp. singular values) of $A_{n}$ can
be approximated by a sampling of $\lambda_1(F)$ (resp. $\sigma_1(F)$)
on a uniform equispaced grid of the domain $G$, and so on until the
last $d_n/\nu$ eigenvalues (resp. singular values) which can be approximated by an equispaced sampling
of $\lambda_s(F)$ (resp. $\sigma_s(F)$) in the domain.

Notice that, in accordance with \eqref{eq:distribution_sv}, $s$ has to be replaced by $r=\min\{s,t\}$ when $F$ is rectangular $s\times t$ and in that case only a distribution in the sense of the singular values can be considered.

Notice that, in accordance with \cite{au2,au1}, any equispaced grid-sequence can be replaced by any asymptotically uniform grid-sequence.
\end{remark}

Here we consider the matrix-sequence $\{A_n\}_n$, where, for each $n$, $A_n$ is a $d_n\times d_n$ matrix as in (\ref{eq:A_Toeplitz}). We select the following groups of examples:
\begin{enumerate}
\item $\nu=2$, $t=s=1$. First, we consider a $2\times 2$ block case where we select a set of scalar valued functions  $f_{i,j}$. For this first group of examples we fix the sizes of the Toeplitz blocks equal to:
\begin{enumerate}
\item[a)] $n_1=\eta$, $n_2=2\eta$;
\item[b)] $n_1=\eta$, $n_2=2\eta+4$;
\item[c)] $n_1=\eta$, $n_2=2\eta+\lceil{\sqrt{\eta}\,\rceil}$.
\end{enumerate}

\item  $\nu=2$, $s=1,$ $t=2$. Here, $A_n$ is a $2\times 2$ block matrix where the Toeplitz blocks are associated with $1\times 2$ rectangular block functions $f_{i,j}$.
The considered dimensions are:
\begin{enumerate}
\item[a)] $n_1=\eta$, $n_2=2\eta$;
\item[b)] $n_1=\eta$, $n_2=2\eta+4$;
\item[c)] $n_1=\eta$, $n_2=2\eta+\lceil{\sqrt{\eta}\,\rceil}$.
\end{enumerate}

\item $\nu=3$, $t=s=2$. Then, we consider a $3\times 3$ block matrix where the Toeplitz blocks are generated by $2\times 2$ matrix-valued functions $f_{i,j}$. Here, we take $n_1=\eta$, $n_2=\frac{\eta}{2}$, $n_3=2\eta -2$.
\end{enumerate}
Note that in all the three items the dimension choices satisfy the assumptions at the beginning of Section \ref{sec:general_block}.
\paragraph{Group 1. $2\times 2$ case with scalar valued $f_{i,j}$}

For the setting $\nu=2$, $t=s=1$ we consider the following trigonometric polynomials
\begin{equation}
\label{gruppo1_t1s1}
\begin{split}
&f_{1,1}(t) =  2-2\cos(t);\quad f_{2,2}(t) =  2-2\cos(t)-6\cos(2t);\\
&
f_{1,2}(t) =  1-{\rm e}^{-\iota t};\quad f_{2,1}(t) = 1-{\rm e}^{\iota t}.
\end{split}
\end{equation}
Then
\[A_n=\begin{bmatrix}
T_{n_1}(f_{1,1}) &  T_{n_1,n_2}(f_{1,2})\\
T_{n_2,n_1}(f_{2,1}) &  T_{n_2}(f_{2,2})
\end{bmatrix}\]
and for all the dimension relation choices 1.a)-1.b)-1.c) we have
\[
\{A_n\}_n\sim_{\sigma}F
\]
with
\begin{equation}
\label{F1}
F=
\left[\begin{array}{c | cc}
f_{1,1} & f_{1,2}& 0 \\
\hline
f_{2,1} & f_{2,2}& 0  \\
0 &0 &f_{2,2}  \\
\end{array}\right].
\end{equation}
Since $F=F^*$, it also hold $\{A_n\}_n\sim_{\lambda}F$.

In order to numerically show the latter asymptotic distributions, we fix $\eta=20,40,80$ for the size choice 1.a)-1.b) and $\eta=25,49,81$ for the choice 1.c). In Figures \ref{fig:nu1t1s1_eig_eta_2eta_204080}, \ref{fig:nu1t1s1_eig_eta_2eta4_204080}  and \ref{fig:nu1t1s1_eig_eta_2etasqrt_254981}
we can observe the comparison between the eigenvalues of $A_n$, denoted by ${\rm eig}$, and the samplings of the eigenvalue function $\lambda_l$ of $F$ over $$\theta_{\eta}=\left\{-\frac{\pi\eta}{\eta+1}+\frac{2j\eta\pi}{(\eta+1)(\eta-1)}, \, j=0,\dots, \eta-1\right\},$$ which is a uniform grid of $\eta$ points over $[-\pi,\pi]$. The analytical expression of the eigenvalue functions is obtained using \textit{MATLAB} symbolic computation.

\begin{figure}[htb]
\begin{subfigure}[c]{.50\textwidth}
\includegraphics[width=\textwidth]{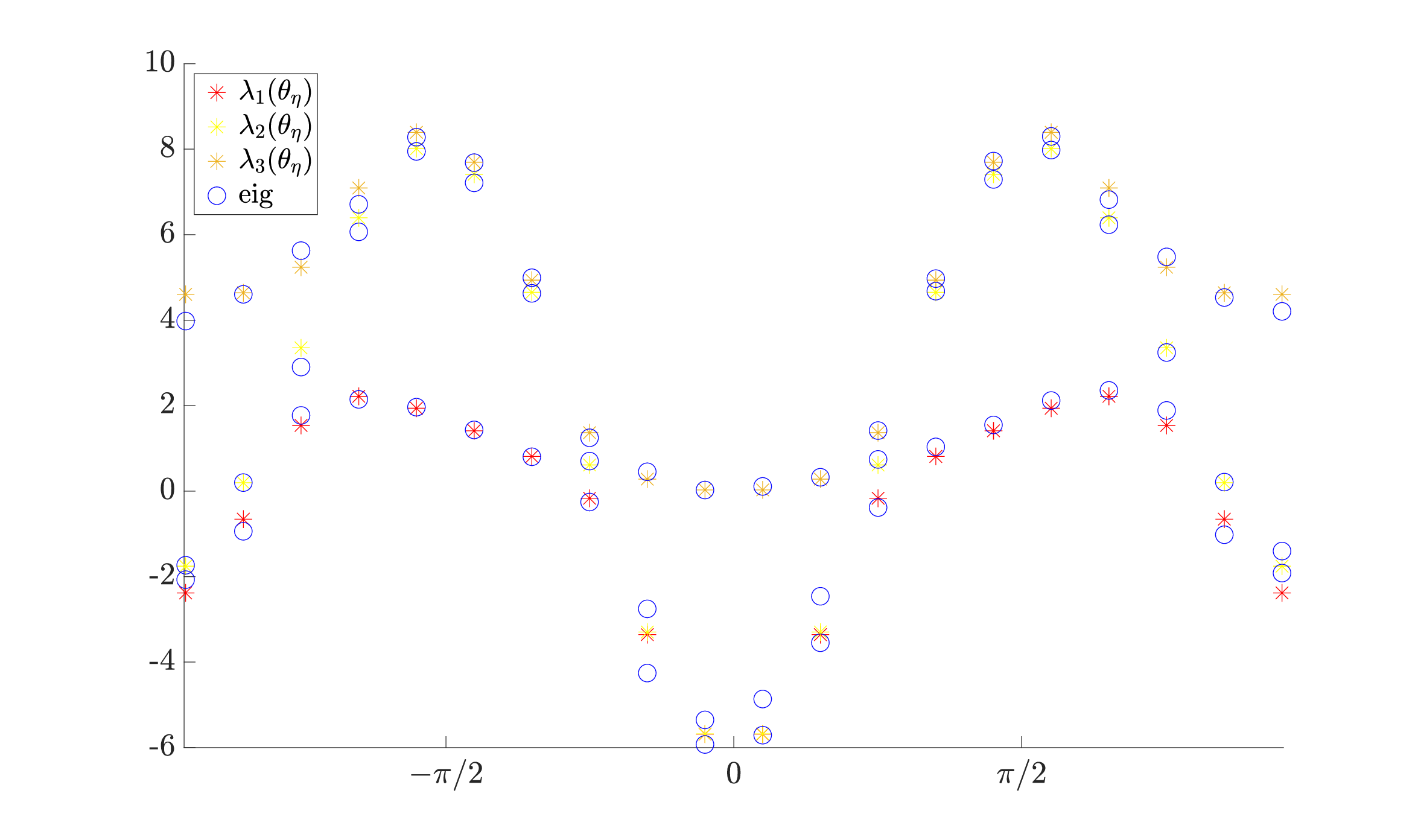}
\end{subfigure}
\begin{subfigure}[c]{.50\textwidth}
\includegraphics[width=\textwidth]{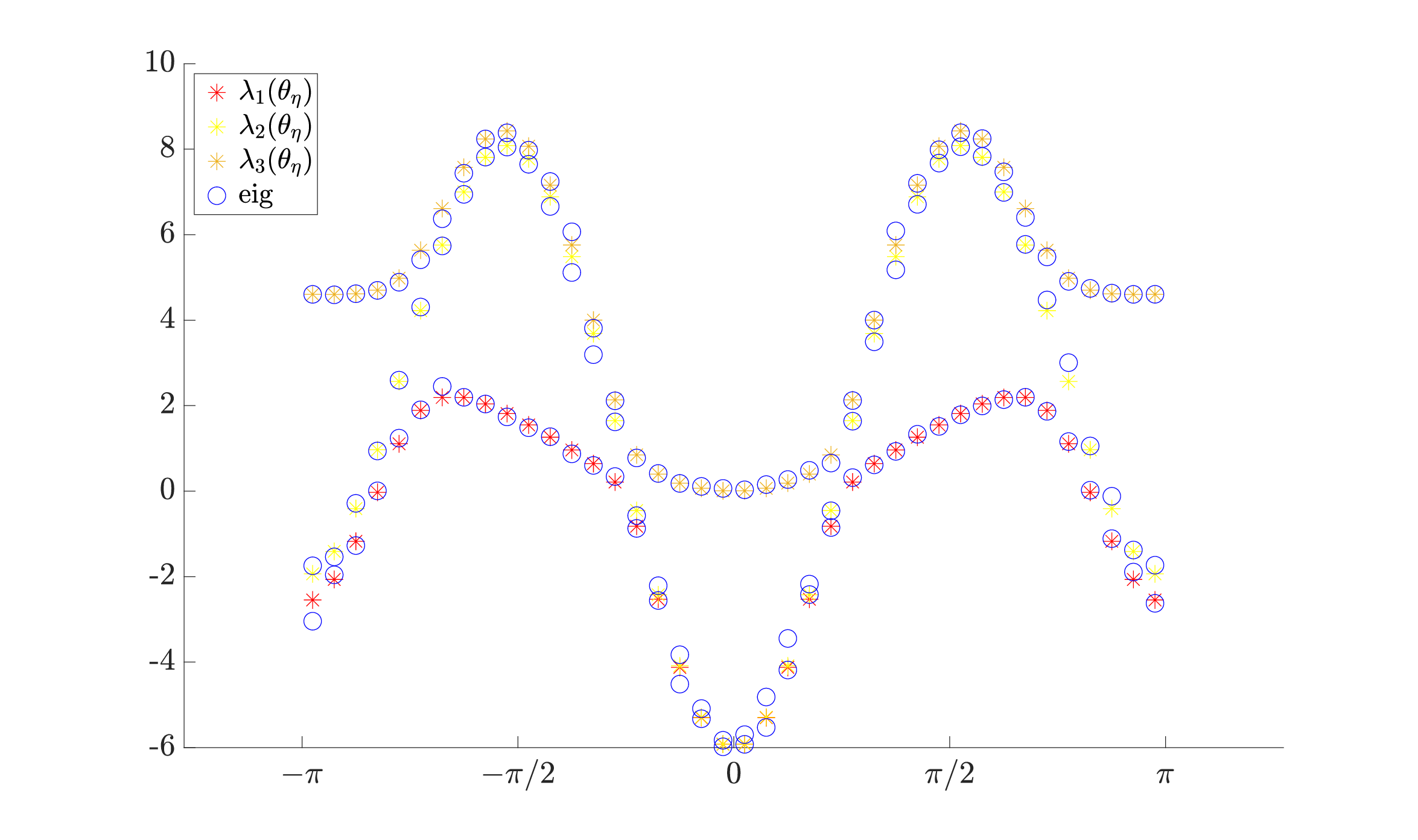}
\end{subfigure}\\
\begin{center}
\begin{subfigure}[c]{.80\textwidth}
\includegraphics[width=\textwidth]{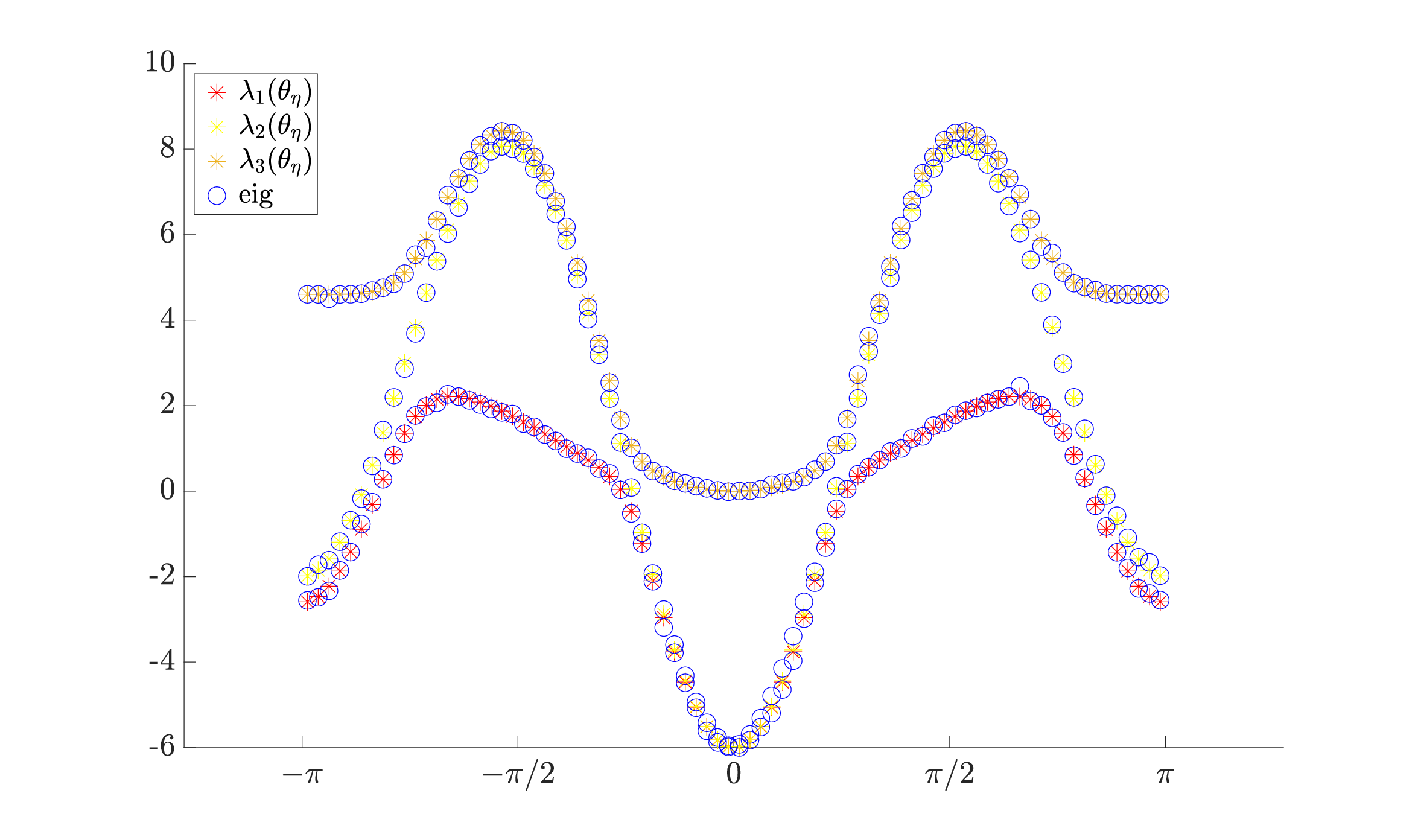}
\end{subfigure}
\end{center}
\caption{Comparison between the eigenvalue functions $\lambda_l(F)$, $l=1,\ldots,3$  and the eigenvalues of $A_n$, for $\nu=2$, $t,s=1$, $n_1=\eta$, $n_2=2\eta$, $\eta= 20, 40, 80 $.}
\label{fig:nu1t1s1_eig_eta_2eta_204080}
\end{figure}

\begin{figure}[htb]
\begin{subfigure}[c]{.50\textwidth}
\includegraphics[width=\textwidth]{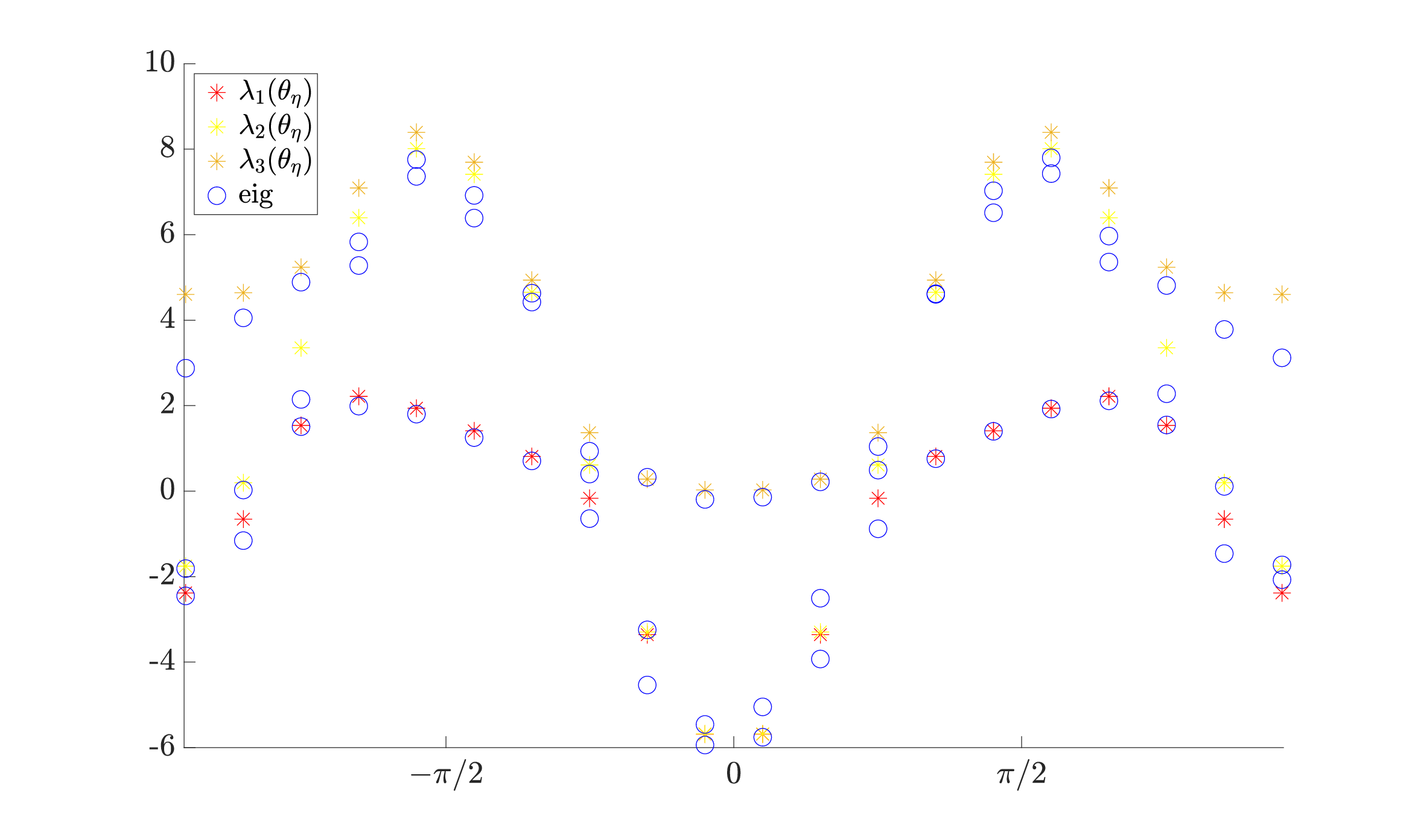}
\end{subfigure}
\begin{subfigure}[c]{.50\textwidth}
\includegraphics[width=\textwidth]{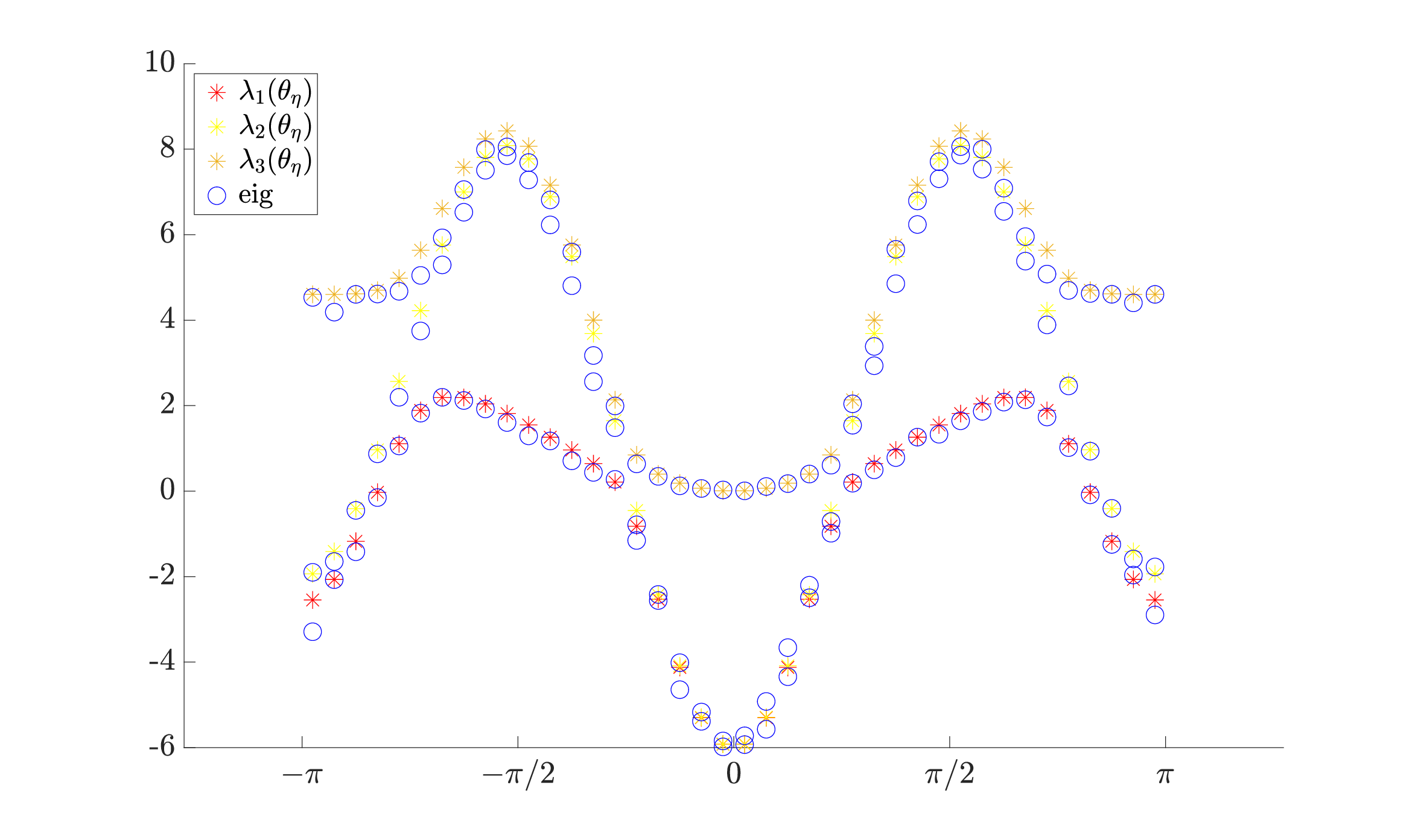}
\end{subfigure}\\
\begin{center}
\begin{subfigure}[c]{.80\textwidth}
\includegraphics[width=\textwidth]{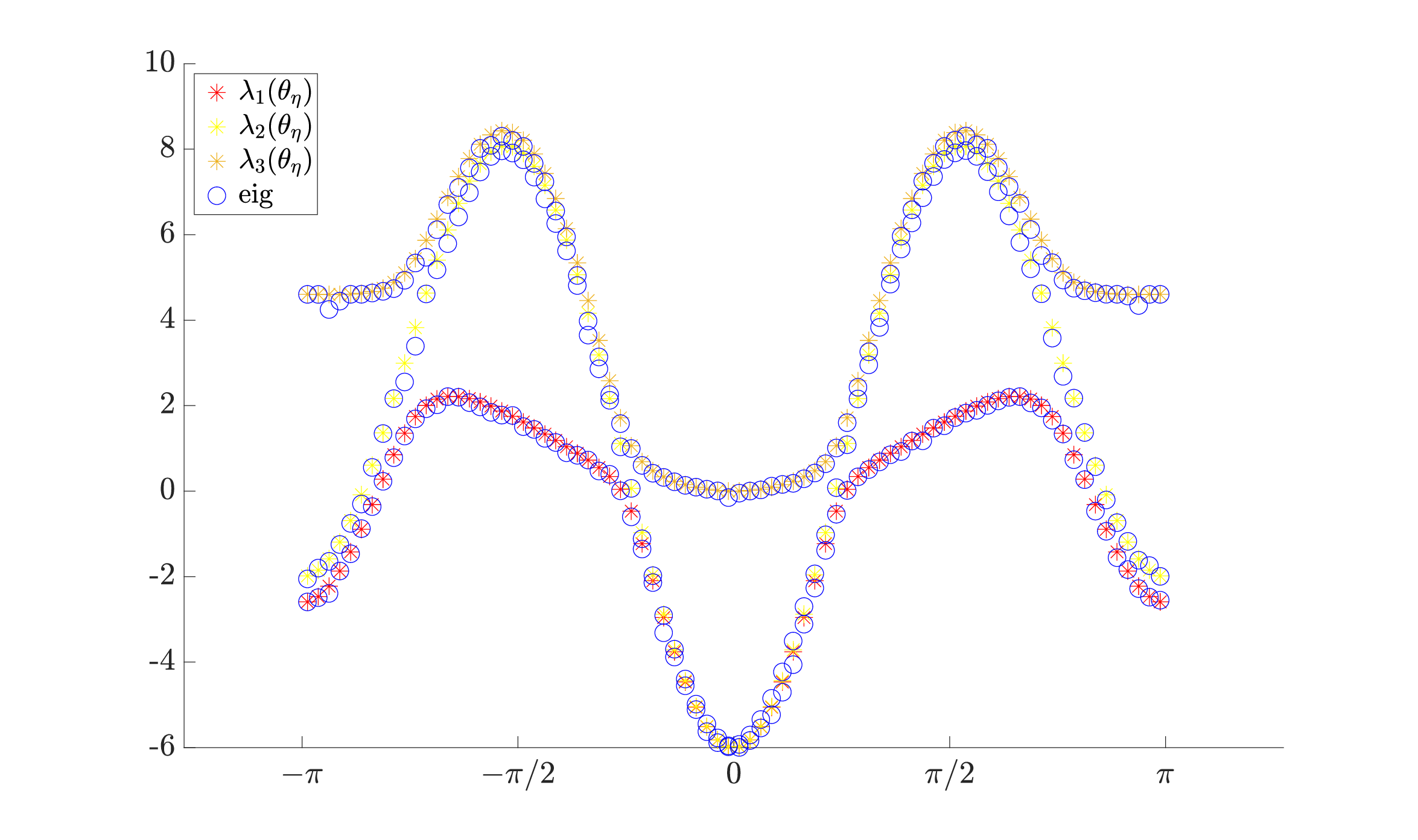}
\end{subfigure}
\end{center}
\caption{Comparison between the eigenvalue functions $\lambda_l(F)$, $l=1,\ldots,3$  and the eigenvalues of $A_n$, for $\nu=2$, $t,s=1$, $n_1=\eta$, $n_2=2\eta+4$, $\eta= 20, 40, 80 $.}
\label{fig:nu1t1s1_eig_eta_2eta4_204080}
\end{figure}

\begin{figure}[htb]
\begin{subfigure}[c]{.50\textwidth}
\includegraphics[width=\textwidth]{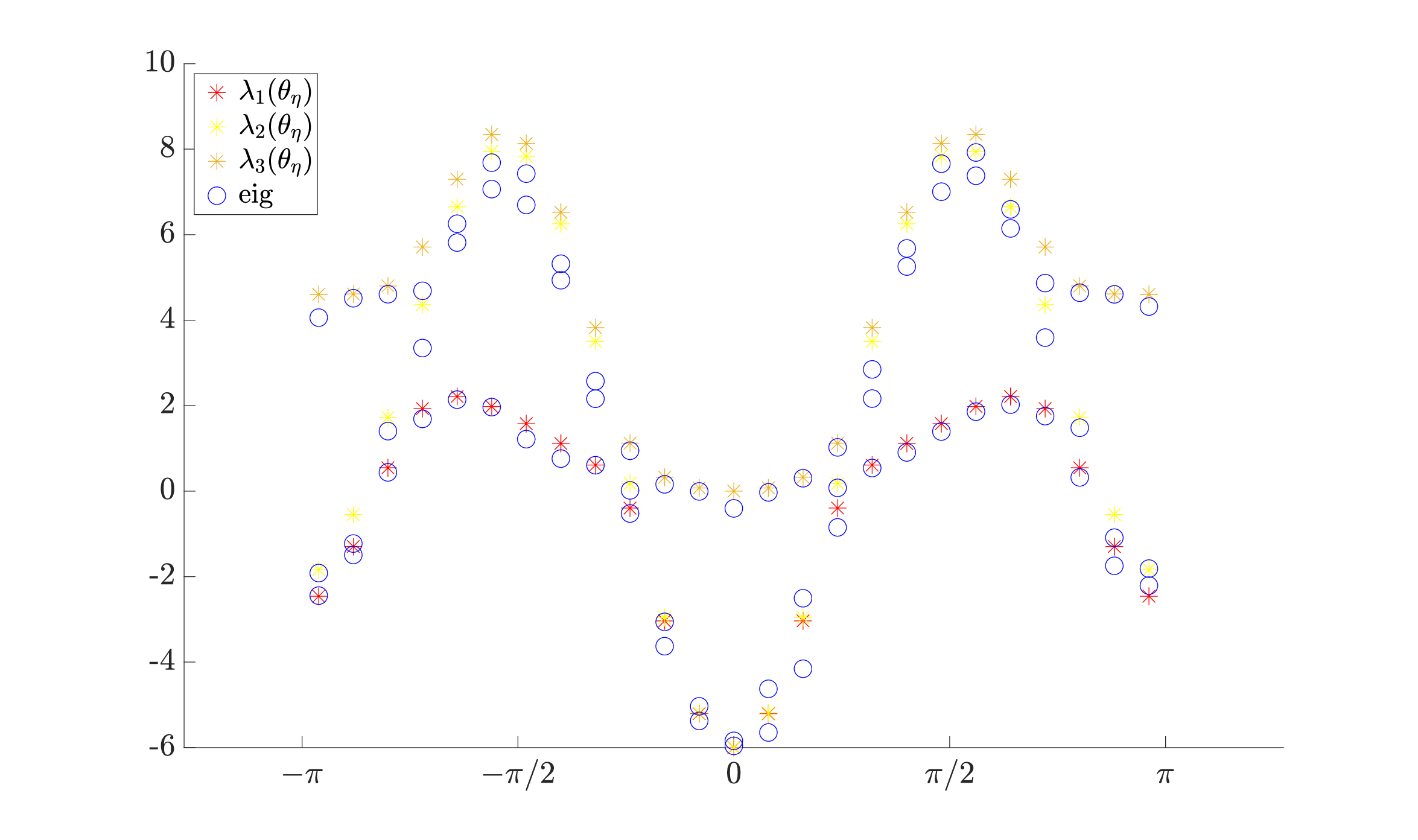}
\end{subfigure}
\begin{subfigure}[c]{.50\textwidth}
\includegraphics[width=\textwidth]{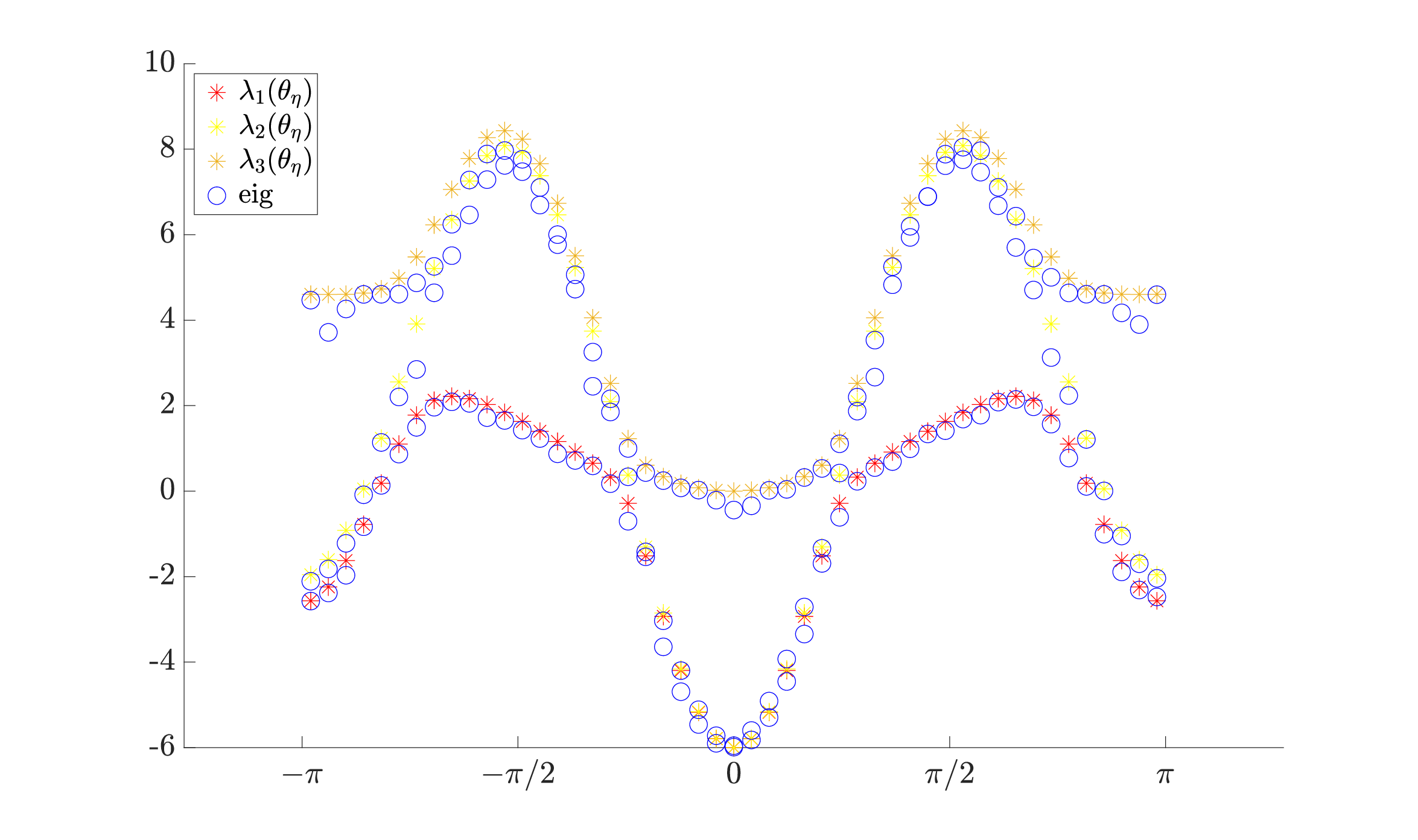}
\end{subfigure}\\
\begin{center}
\begin{subfigure}[c]{.80\textwidth}
\includegraphics[width=\textwidth]{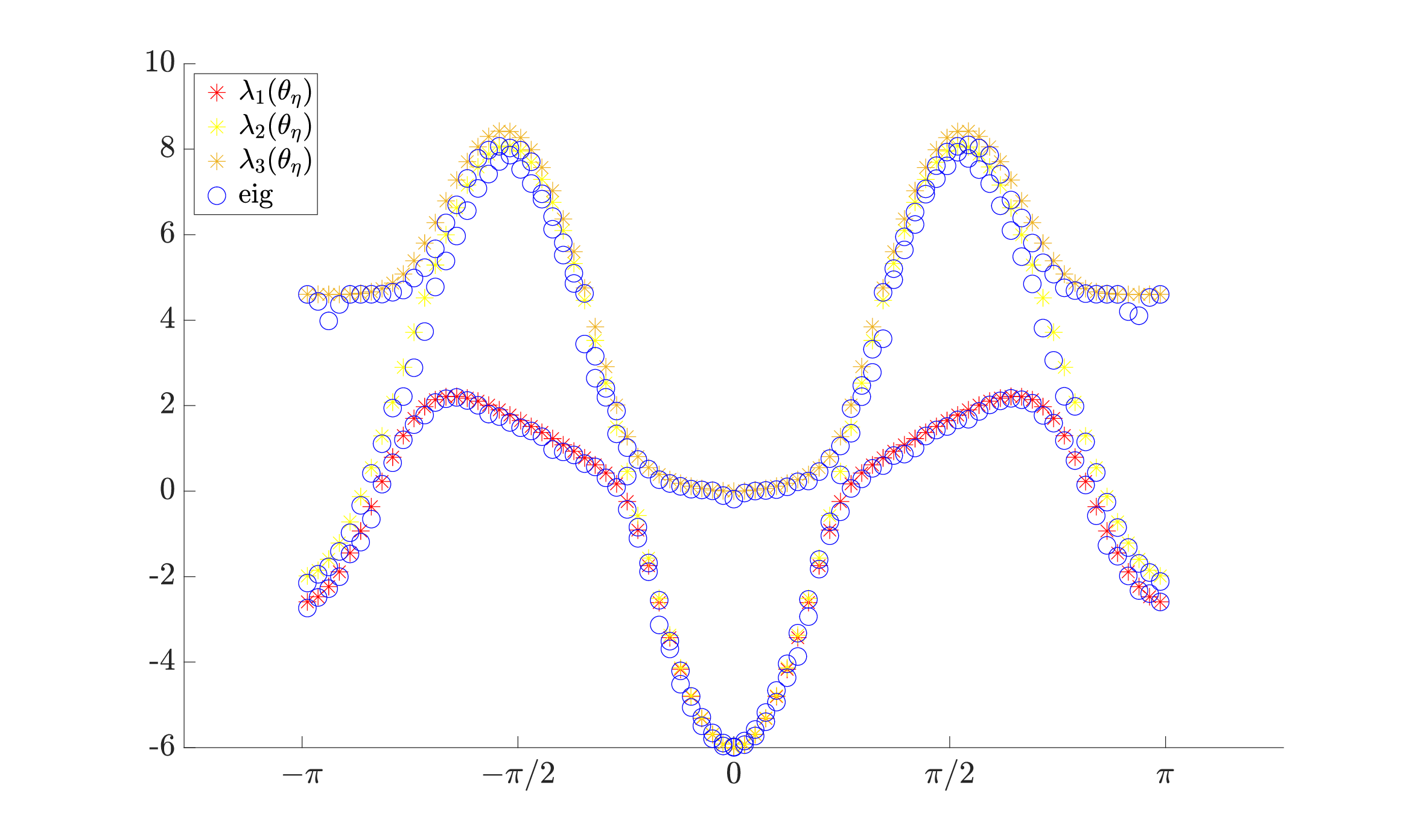}
\end{subfigure}
\end{center}
\caption{Comparison between the eigenvalue functions $\lambda_l(F)$, $l=1,\ldots,3$  and the eigenvalues of $A_n$, for $\nu=2$, $t,s=1$, $n_1=\eta$, $n_2=2\eta+\lceil{\sqrt{\eta}\,\rceil}$, $\eta= 25, 49, 81 $.}
\label{fig:nu1t1s1_eig_eta_2etasqrt_254981}\end{figure}

 The graph comparison highlights  that each eigenvalue function approximates roughly $1/3$ of the eigenvalues of $A_n$ and the estimation becomes more precise with $n\rightarrow \infty$  as indicated by Remark \ref{rmrk:distributions}, numerically confirming the spectral distribution result in Theorem \ref{almost equal-blocks-basic}. Similarly, the singular value distribution can be numerically tested with a comparison between the singular values of $A_n$, denoted by $\sigma(A),$ and the singular value functions $\sigma_l(F)$, $l=1,\ldots,3$ sampled over $\theta_{\eta}$. The approximation result is shown in Figure \ref{fig:nu1t1s1_sing_eta_2etasqrt_256} for the dimension choice 1.c) and $\eta= 256$.

 \begin{figure}[htb]
\begin{center}
\includegraphics[width=0.8\textwidth]{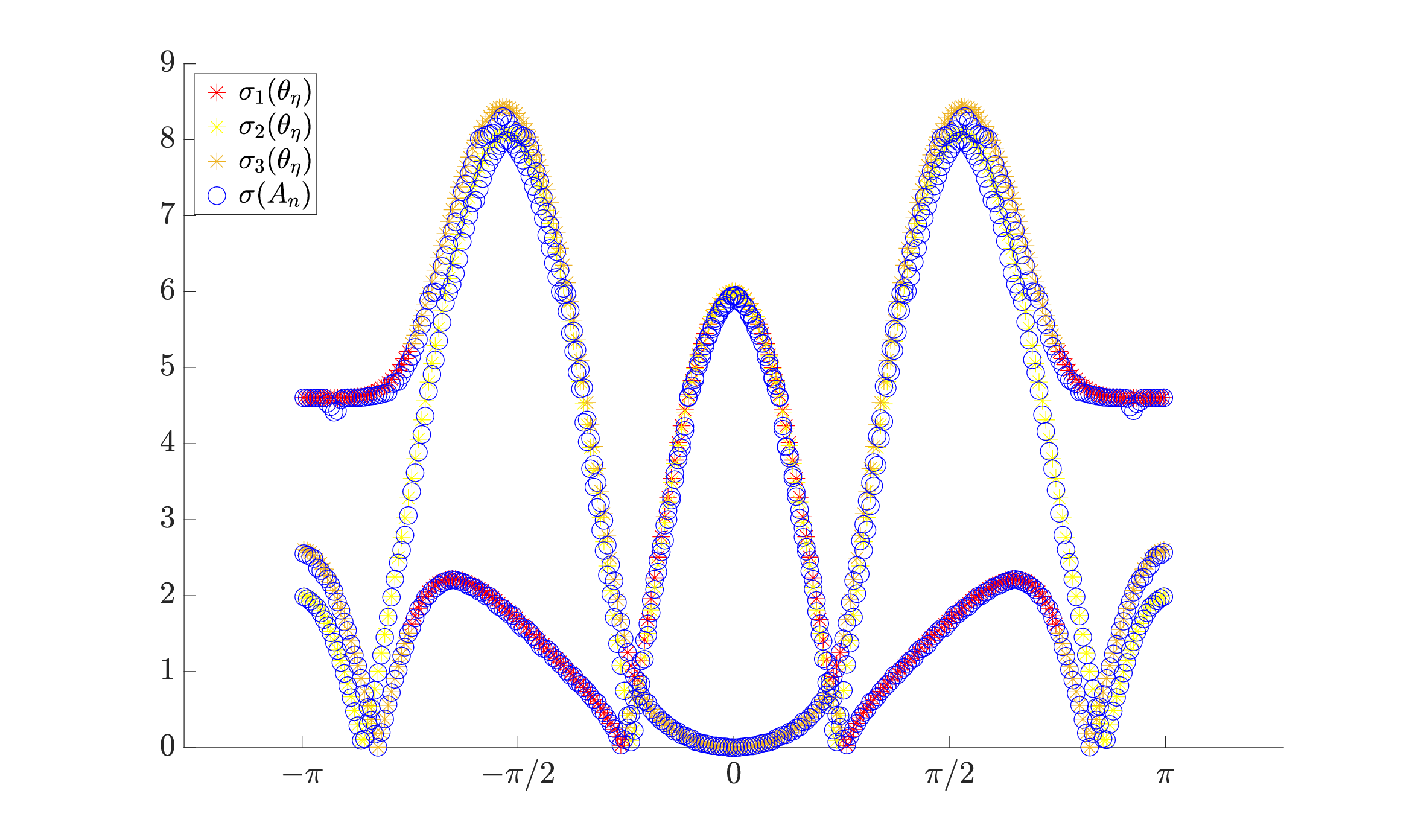}
\end{center}
\caption{Comparison between the singular value functions $\sigma_l(F)$, $l=1,\ldots,3$  and the singular values of $A_n$, for $\nu=2$, $t,s=1$, $n_1=\eta$, $n_2=2\eta+\lceil{\sqrt{\eta}\,\rceil}$, $\eta= 256$.}
\label{fig:nu1t1s1_sing_eta_2etasqrt_256}
\end{figure}

We conclude this group of examples with a test on the singular value distribution in a case where one of the $f_{i,j}$  is an $L^1$ function and moreover the resulting $F$ has the same structure of (\ref{F1}) but $F\neq F^*$. In particular, we take
\begin{equation}
\label{gruppo1_t1s1_singolari}
\begin{split}
&f_{1,1}(t) =  \theta^2; \quad f_{2,2}(t) =  2-2\cos(t)-6\cos(2t);\\
&
f_{1,2}(t) =  1-{\rm e}^{-\iota t};\quad f_{2,1}(t) = f_{1,2}(t).
\end{split}
\end{equation}
We observe the numerical confirmation of $\{A_n\}_n\sim_\sigma F$ in Figure \ref{fig:nu1t1s1_sing_eta_2etasqrt_254981}, for the dimension choice 1.c) and $\eta= 25,49,81$.

\begin{figure}[htb]
\begin{subfigure}[c]{.50\textwidth}
\includegraphics[width=\textwidth]{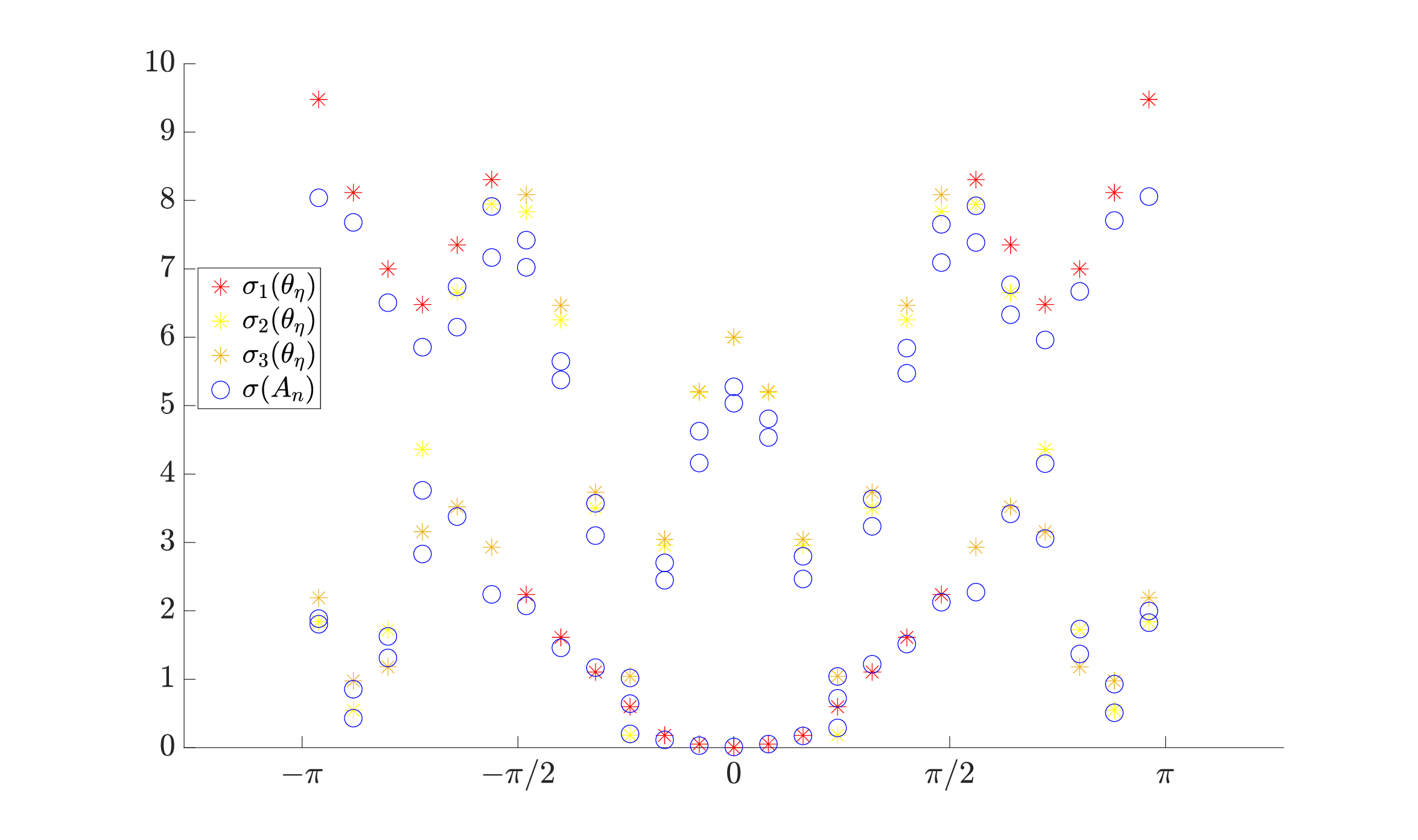}
\end{subfigure}
\begin{subfigure}[c]{.50\textwidth}
\includegraphics[width=\textwidth]{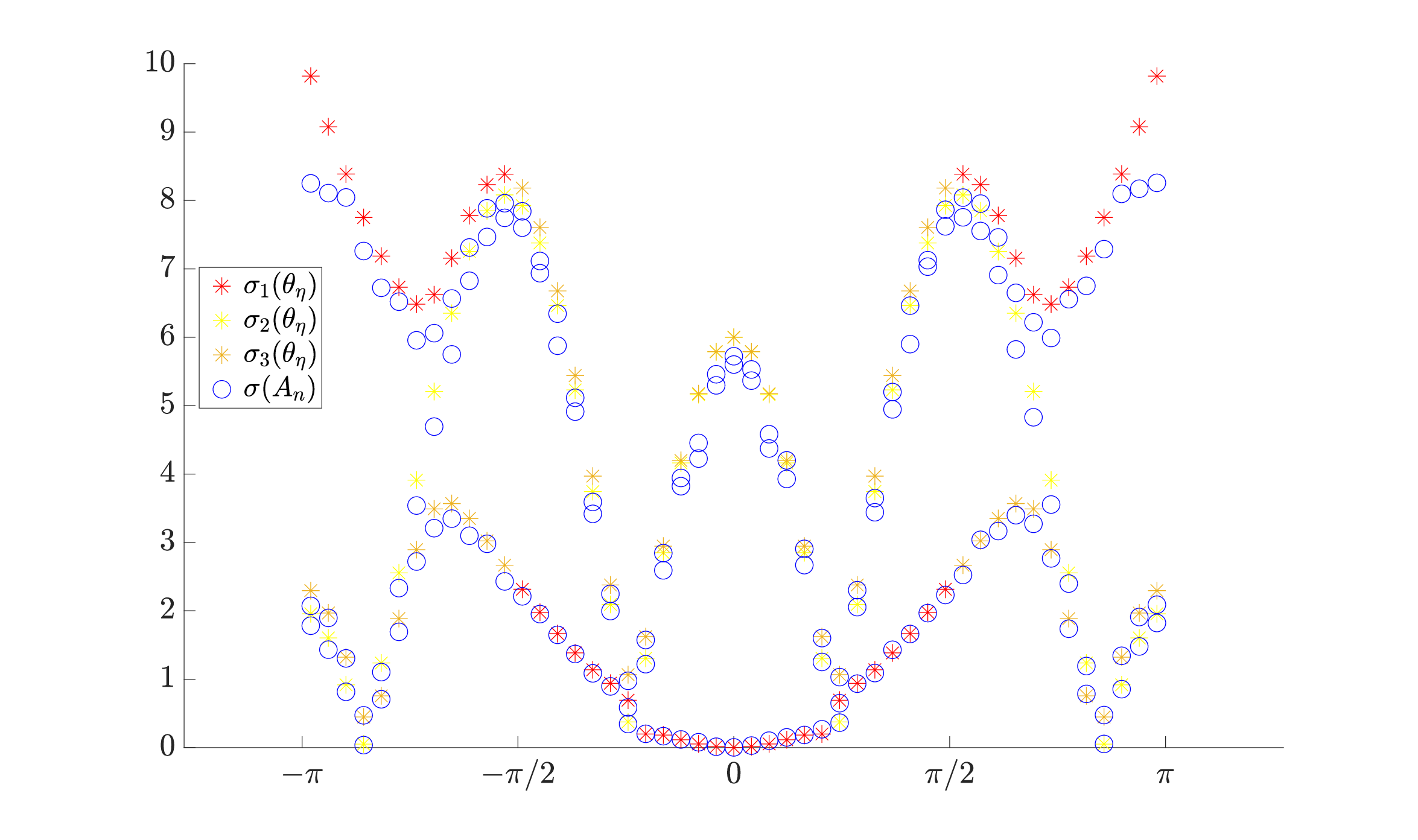}
\end{subfigure}\\
\begin{center}
\begin{subfigure}[c]{.80\textwidth}
\includegraphics[width=\textwidth]{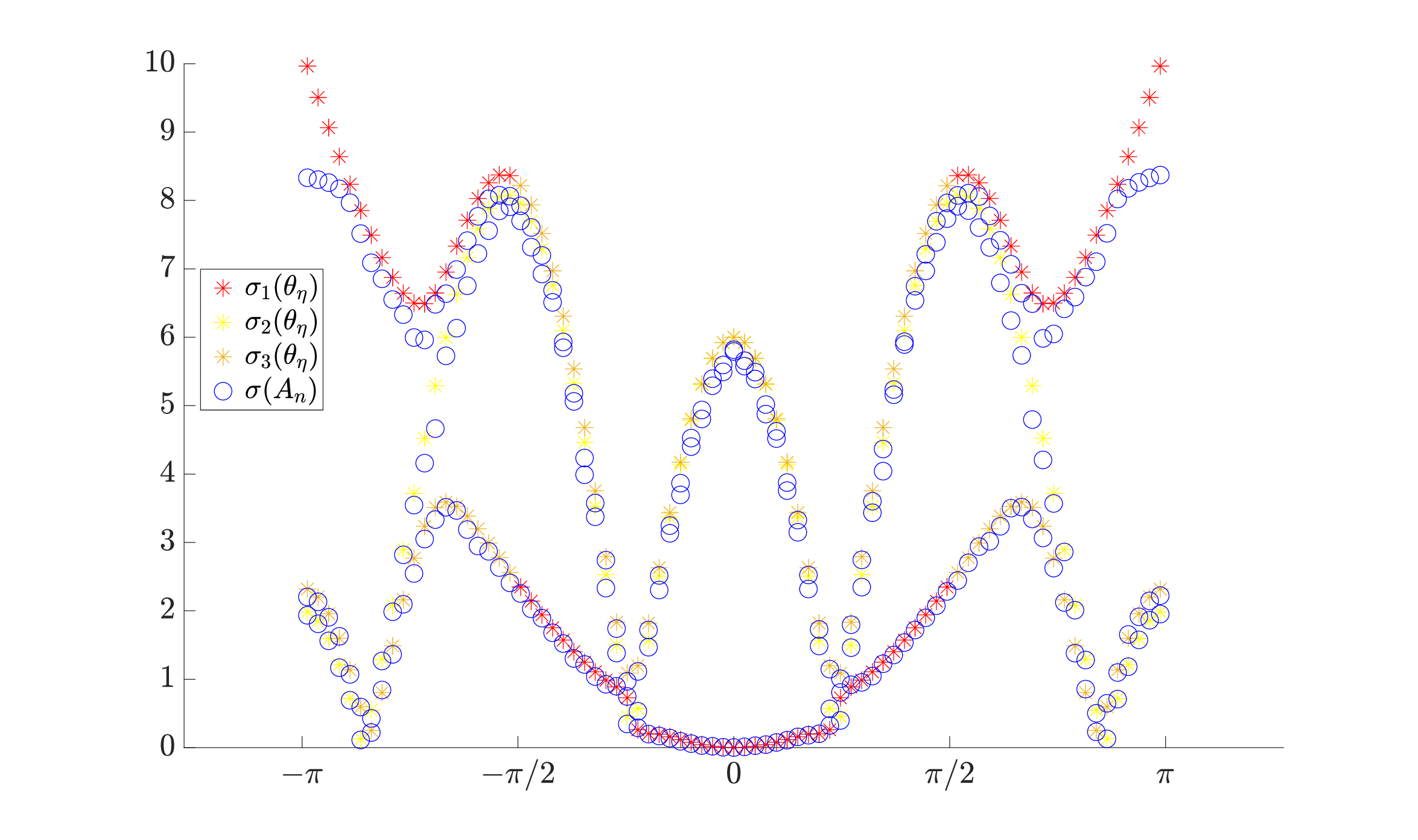}
\end{subfigure}
\end{center}
\caption{Comparison between the singular value functions $\sigma_l(F)$, $l=1,\ldots,3$  and the singular values of $A_n$, for $\nu=2$, $t,s=1$, $n_1=\eta$, $n_2=2\eta+\lceil{\sqrt{\eta}\,\rceil}$, $\eta= 25, 49, 81 $.}
\label{fig:nu1t1s1_sing_eta_2etasqrt_254981}\end{figure}

Moreover, we demonstrate that the cardinality of the singular values that deviate from the behaviour predicted by  $\sigma_l(F)$ is $o(d_n)$. Precisely, we fix the tolerance $h= 10^{-1}$, we order both the  singular values and the samplings of $\sigma_l(F)$, $l=1,\dots,3$ in an increasing order and  compute the vector $out^{(\sigma)}$ as the element-wise absolute difference between these two ordered vectors. Table \ref{tab:outliers_sigma_group1} shows the ratio between $\#\{out^{(\sigma)}\ge h\}$ and the size of $A_n$ tends to 0 as $n$ increases, confirming that the number of outliers that do not mimic the singular value functions is $o(d_n)$, as predicted by relation $\{A_n\}_n\sim_\sigma F$.

\begin{table}[htb]
\begin{center}
\begin{tabular}{|c|c|}
\hline
$\eta$ &$\#\{out^{(\sigma)}\ge h\}$/$d_n$ \\
\hline\hline
81&  0.5833\\
144&  0.5135\\
225& 0.4362\\
324& 0.3949\\
441& 0.3638\\
576& 0.3191\\
729& 0.2525\\
900&  0.1564\\
\hline
\end {tabular}
\caption{Ratio between the cardinality of $\sigma_i(A_n)$ that do not behave as the sampling $\sigma_i(F)$ and the dimension of $A_n$, for $\eta = (3x)^2$, $x=3,\dots,10$.}\label{tab:outliers_sigma_group1}
\end{center}
\end{table}

\paragraph{Group 2. $2\times 2$ case with rectangular matrix-valued $f_{i,j}$}

For the setting $\nu=2$, $s=1$, $t=2$ we consider the following rectangular matrix-valued functions

\begin{equation}
\label{gruppo2_t2s1}
\begin{split}
&f_{1,1}(t) =  (2-2\cos(t), 4+6\cos(2t));\quad f_{2,2}(t) =  (3+2\cos(t),4+6\cos(t)-2\cos(2t));\\
&
f_{1,2}(t) =  (1+{\rm e}^{\iota t}, 1-{\rm e}^{-\iota t});\quad f_{2,1}(t) = f_{1,2}(t).
\end{split}
\end{equation}

Then,

\[A_n=\begin{bmatrix}
T_{n_1}(f_{1,1}) &  T_{n_1,n_2}(f_{1,2})\\
T_{n_2,n_1}(f_{2,1}) &  T_{n_2}(f_{2,2})
\end{bmatrix}\]
and for all the dimension relation choices 2.a)-2.b)-2.c) we have
\[\{A_n\}_n\sim_{\sigma}F,\]
with
\begin{equation}
\label{F2}
F=
\left[\begin{array}{c | cc}
f_{1,1} & f_{1,2}& 0 \\
\hline
f_{2,1} & f_{2,2}& 0  \\
0 &0 &f_{2,2}  \\
\end{array}\right].
\end{equation}

In order to numerically show the latter asymptotic distributions, we fix $\eta=20, 40, 80$ for the size choice 2.a)-2.b) and $\eta=25, 49, 81$ for the choice 2.c). In Figures \ref{fig:nu2t2s1_eig_eta_2eta_204080}, \ref{fig:nu2t2s1_eig_eta_2eta4_204080} and \ref{fig:nu2t2s1_eig_eta_2etasqrt_254981} we can observe the comparison between the singular values of $A_n$, denoted by ${\sigma(A_n)}$, and the samplings of the singular value functions $\sigma_l$ of $F$ over $\theta_{\eta}$ which, as above, is a uniform grid of $\eta$ points over $[-\pi,\pi]$. The analytical expression of the singular value functions is obtained using \textit{MATLAB} symbolic computation from equation \eqref{F2}.

\begin{figure}[htb]
\begin{subfigure}[c]{.50\textwidth}
\includegraphics[width=\textwidth]{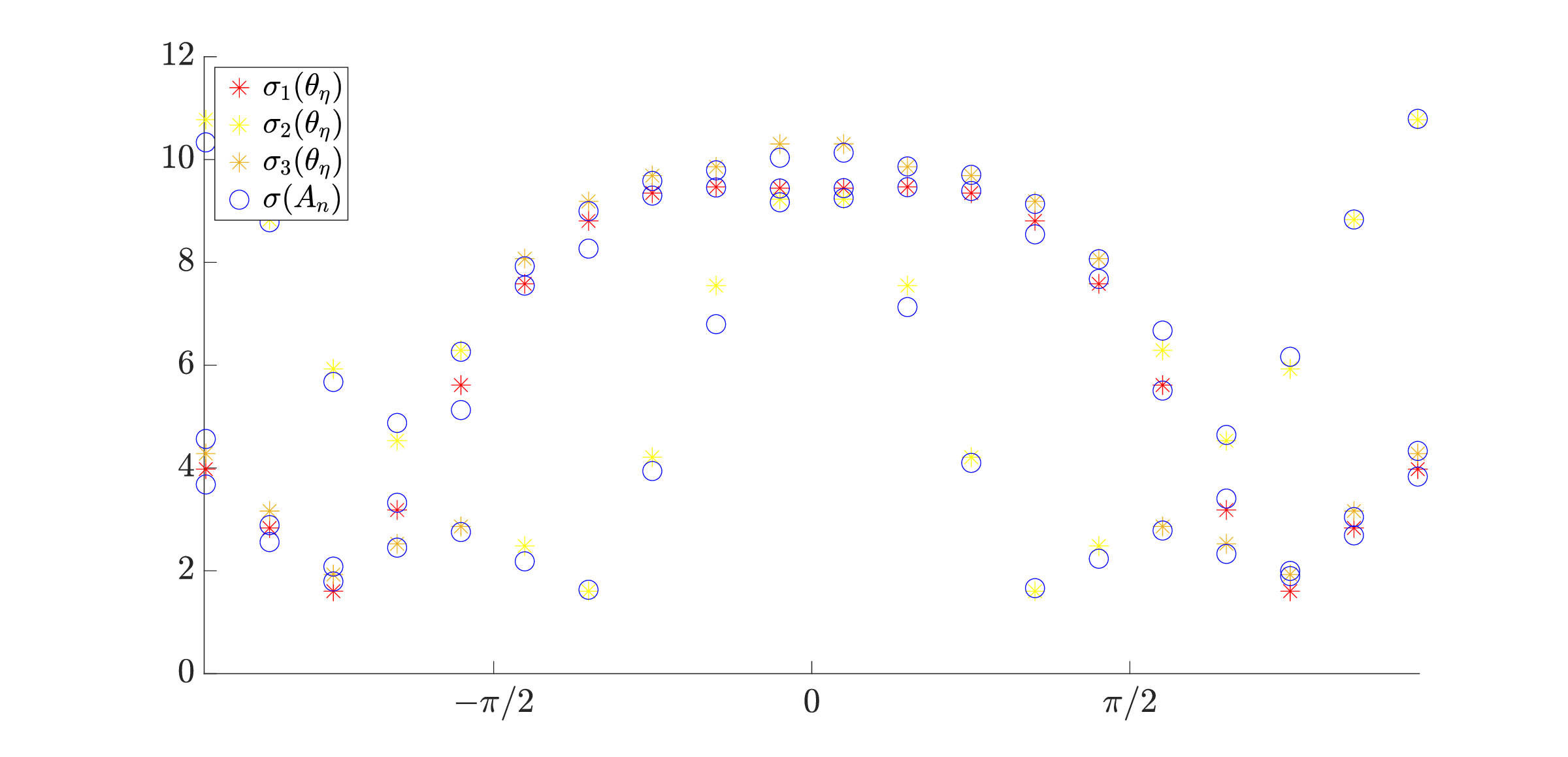}
\end{subfigure}
\begin{subfigure}[c]{.50\textwidth}
\includegraphics[width=\textwidth]{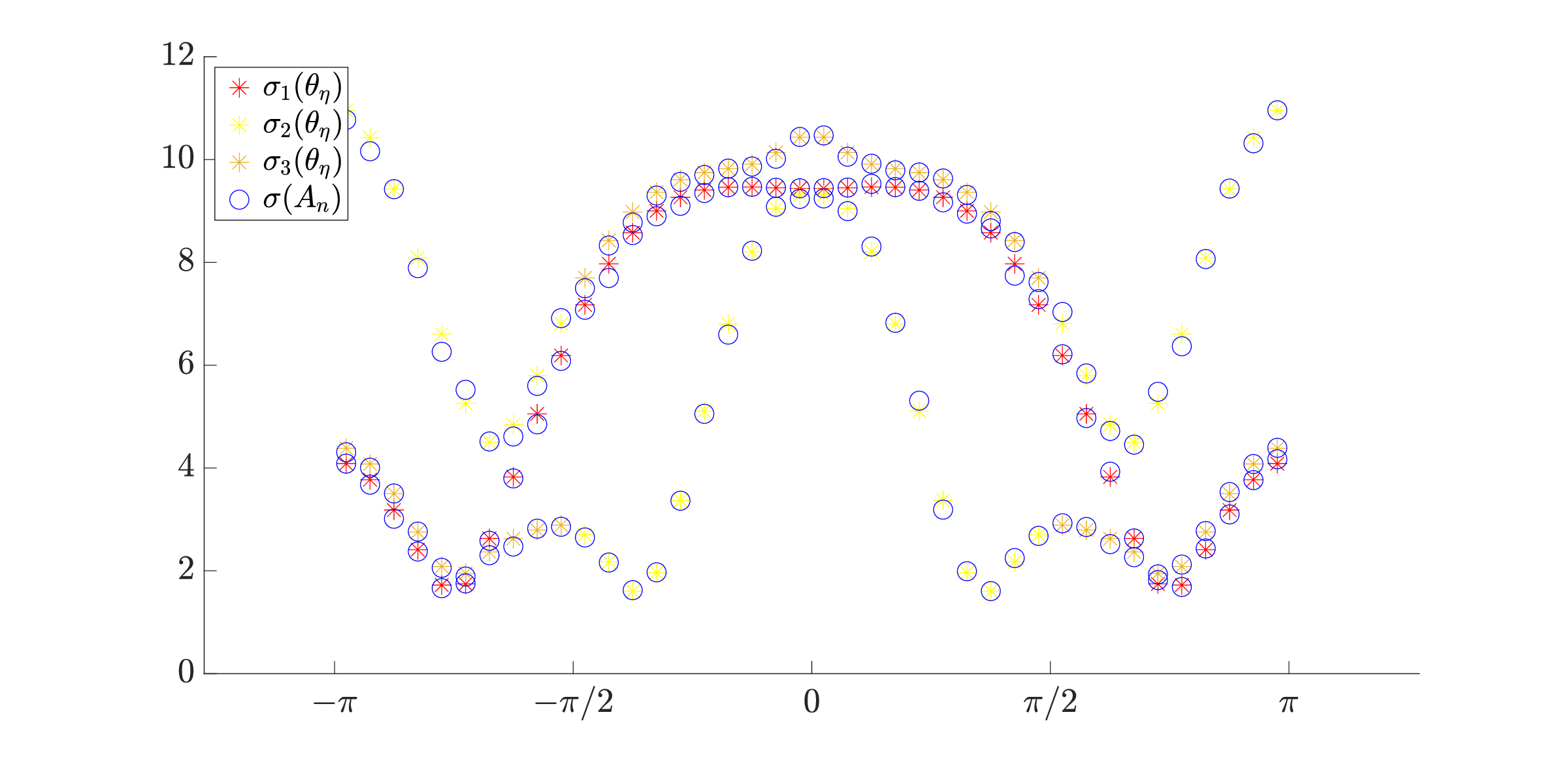}
\end{subfigure}\\
\begin{center}
\begin{subfigure}[c]{.80\textwidth}
\includegraphics[width=\textwidth]{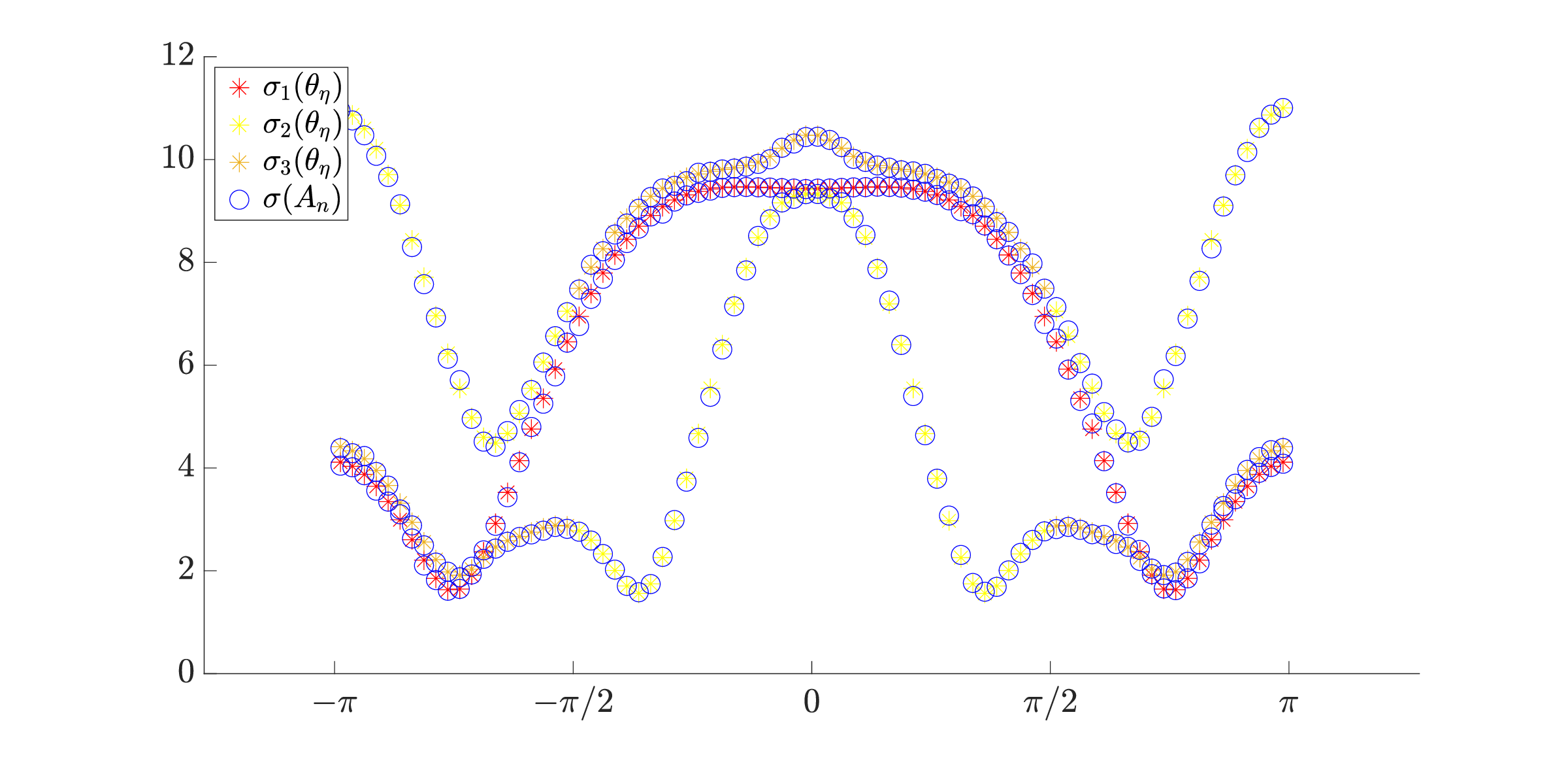}
\end{subfigure}
\end{center}
\caption{Comparison between the singular value functions $\sigma_l(F)$, $l=1,\ldots,3$  and the singular values of $A_n$, for $\nu=2$, $t=2,s=1$, $n_1=\eta$, $n_2=2\eta$, $\eta= 20, 40, 80 $.}
\label{fig:nu2t2s1_eig_eta_2eta_204080}
\end{figure}

\begin{figure}[htb]
\begin{subfigure}[c]{.50\textwidth}
\includegraphics[width=\textwidth]{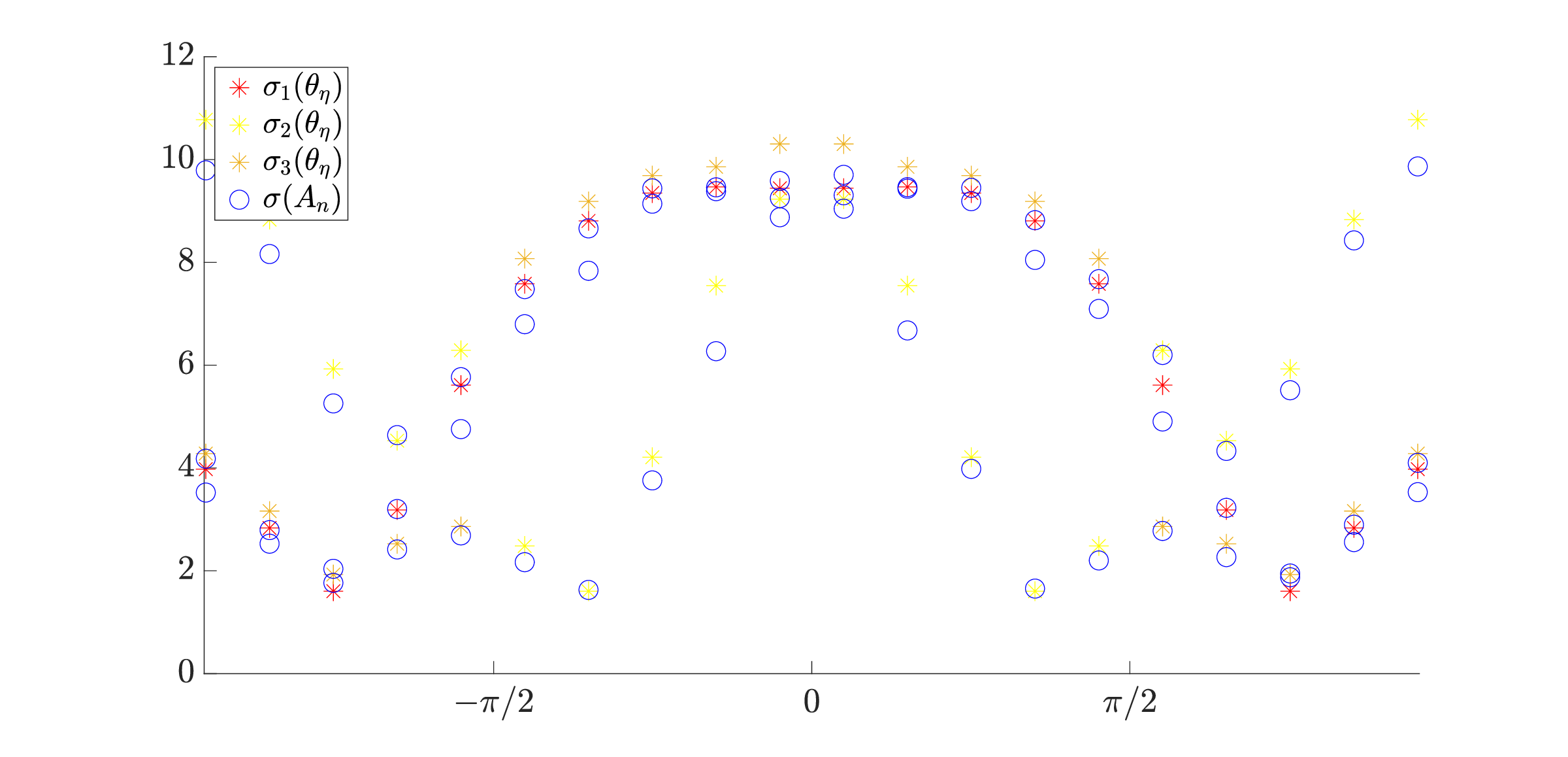}
\end{subfigure}
\begin{subfigure}[c]{.50\textwidth}
\includegraphics[width=\textwidth]{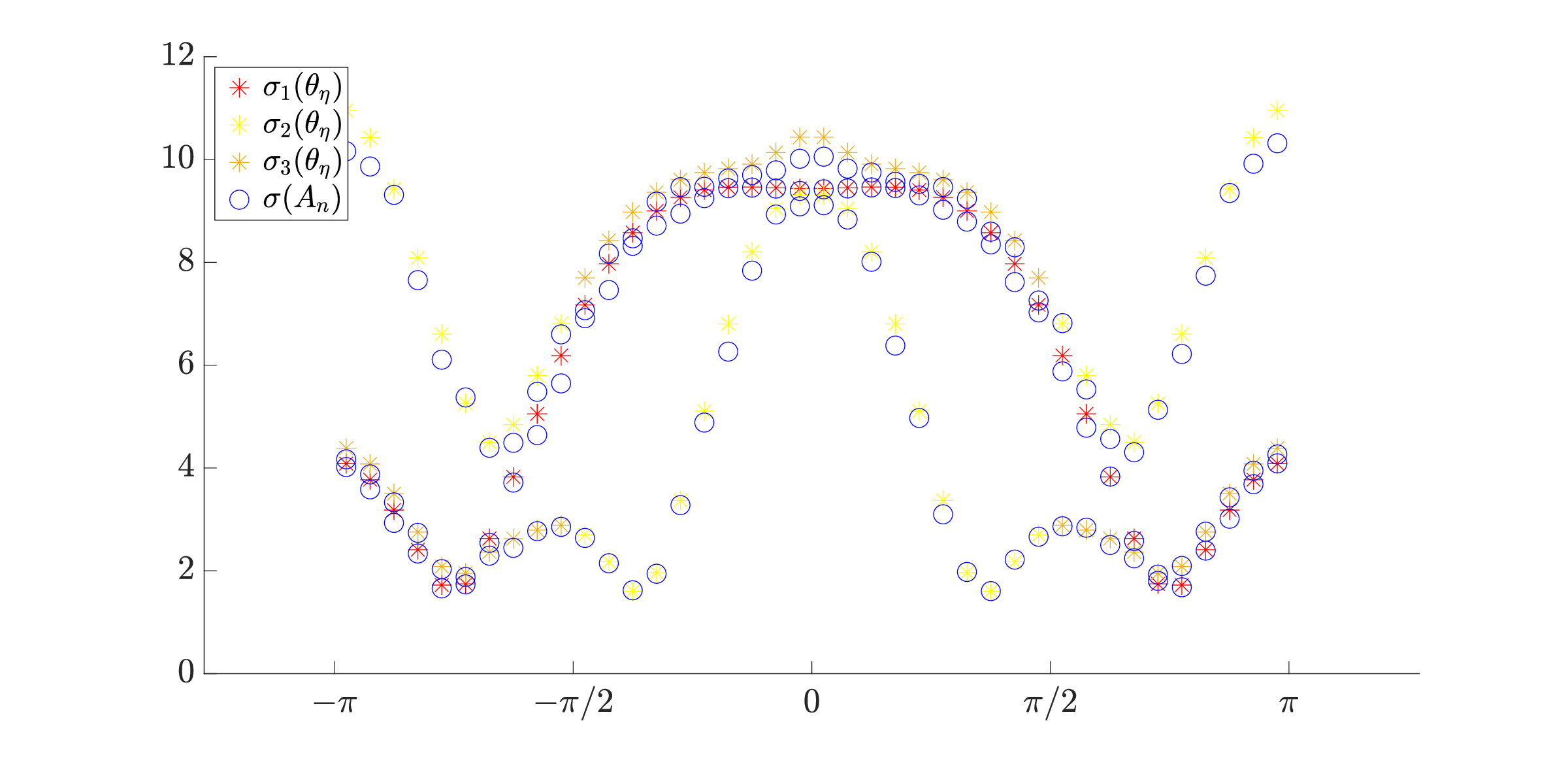}
\end{subfigure}\\
\begin{center}
\begin{subfigure}[c]{.80\textwidth}
\includegraphics[width=\textwidth]{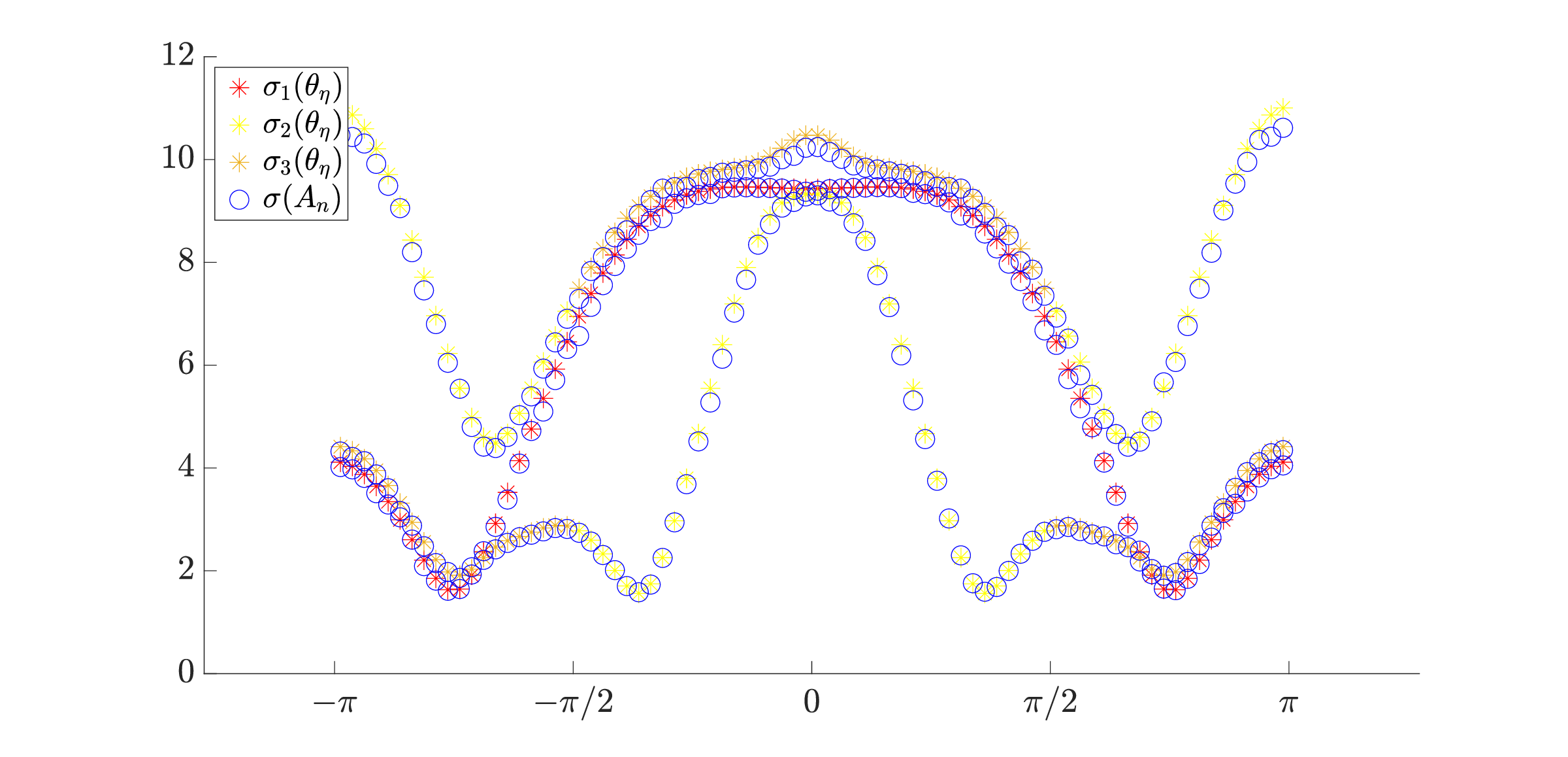}
\end{subfigure}
\end{center}
\caption{Comparison between the singular value functions $\sigma_l(F)$, $l=1,\ldots,3$  and the singular values of $A_n$, for $\nu=2$, $t=2,s=1$, $n_1=\eta$, $n_2=2\eta+4$, $\eta= 20, 40, 80 $.}
\label{fig:nu2t2s1_eig_eta_2eta4_204080}
\end{figure}

\begin{figure}[htb]
\begin{subfigure}[c]{.50\textwidth}
\includegraphics[width=\textwidth]{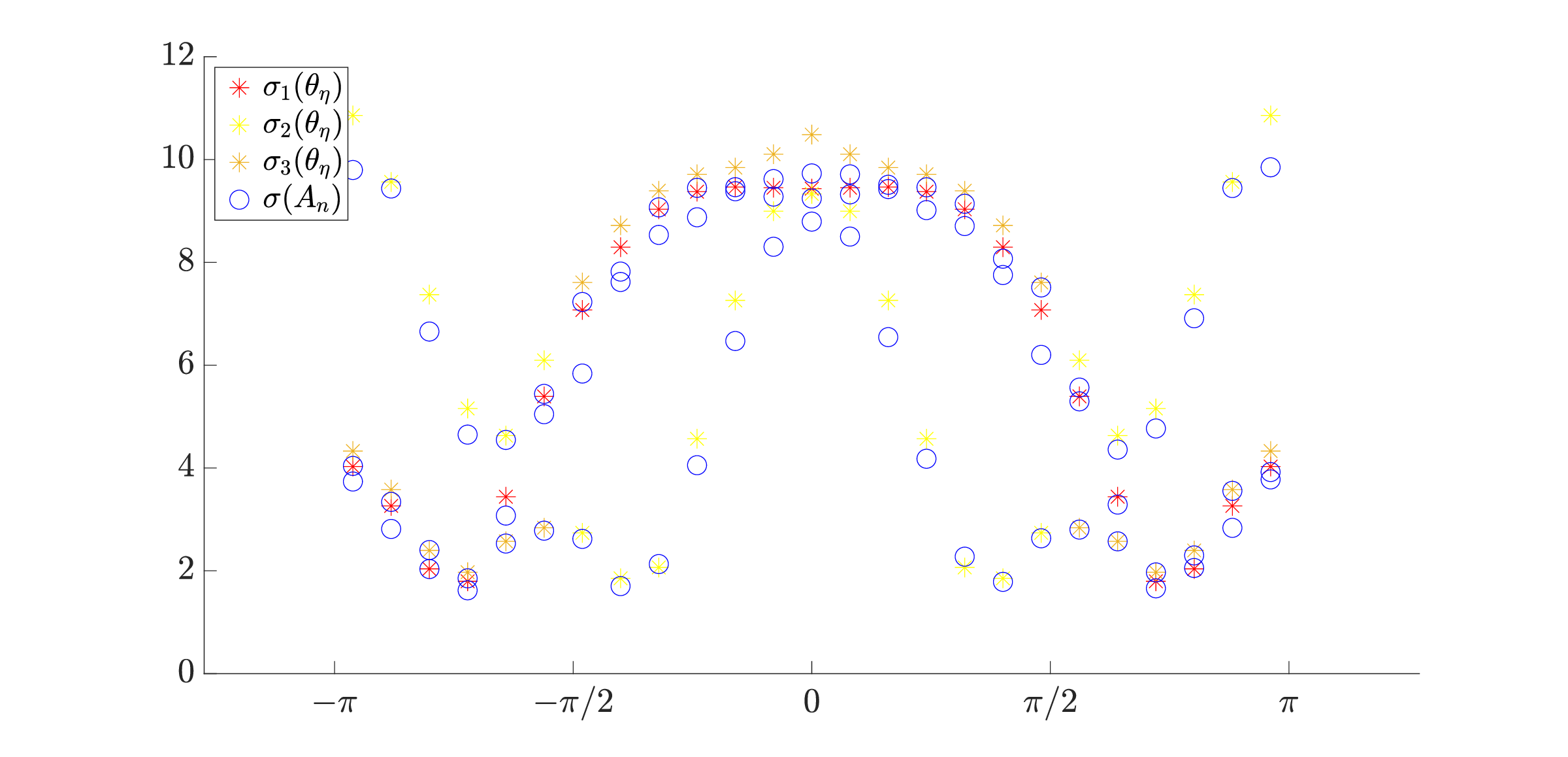}
\end{subfigure}
\begin{subfigure}[c]{.50\textwidth}
\includegraphics[width=\textwidth]{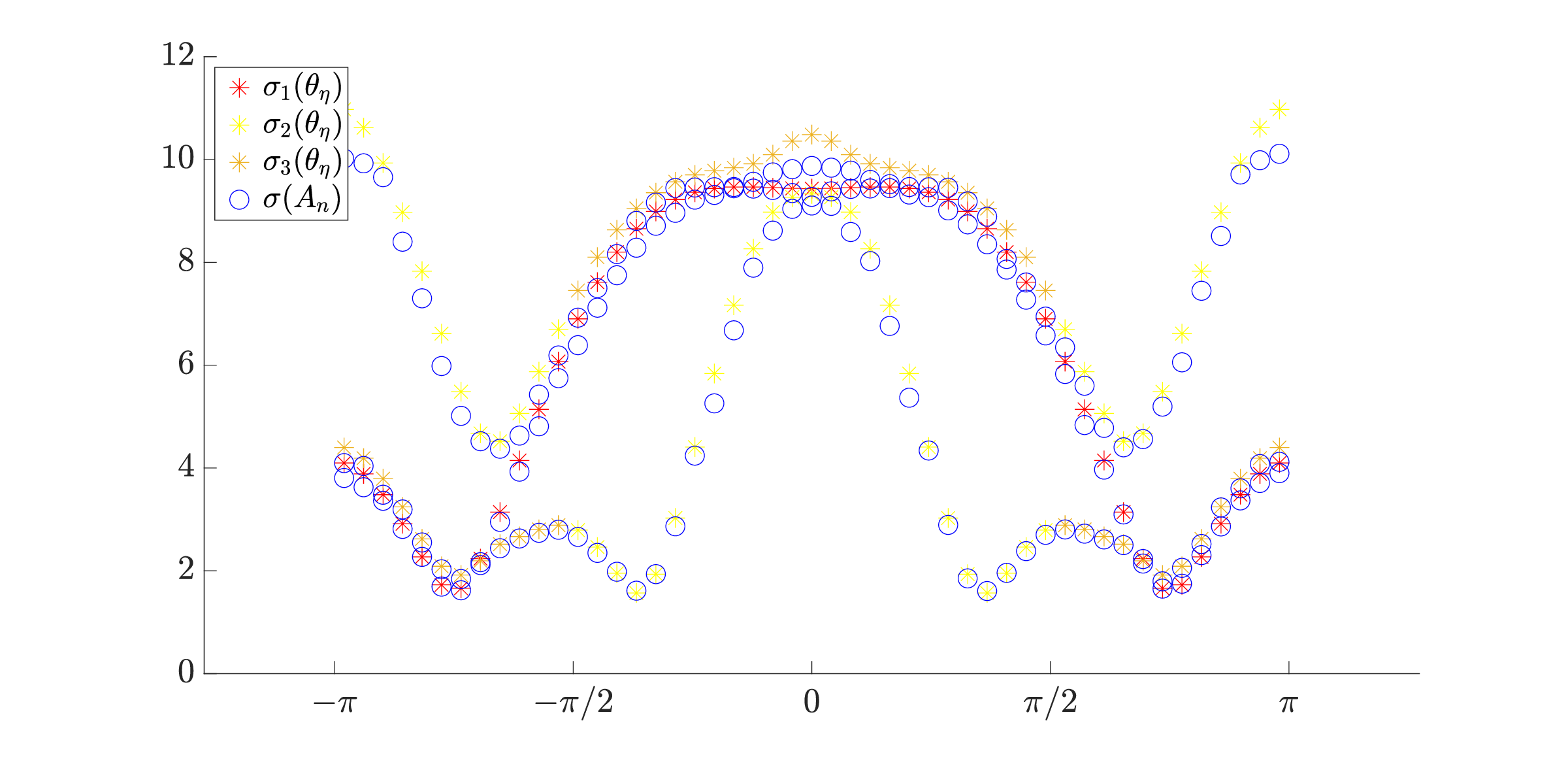}
\end{subfigure}\\
\begin{center}
\begin{subfigure}[c]{.80\textwidth}
\includegraphics[width=\textwidth]{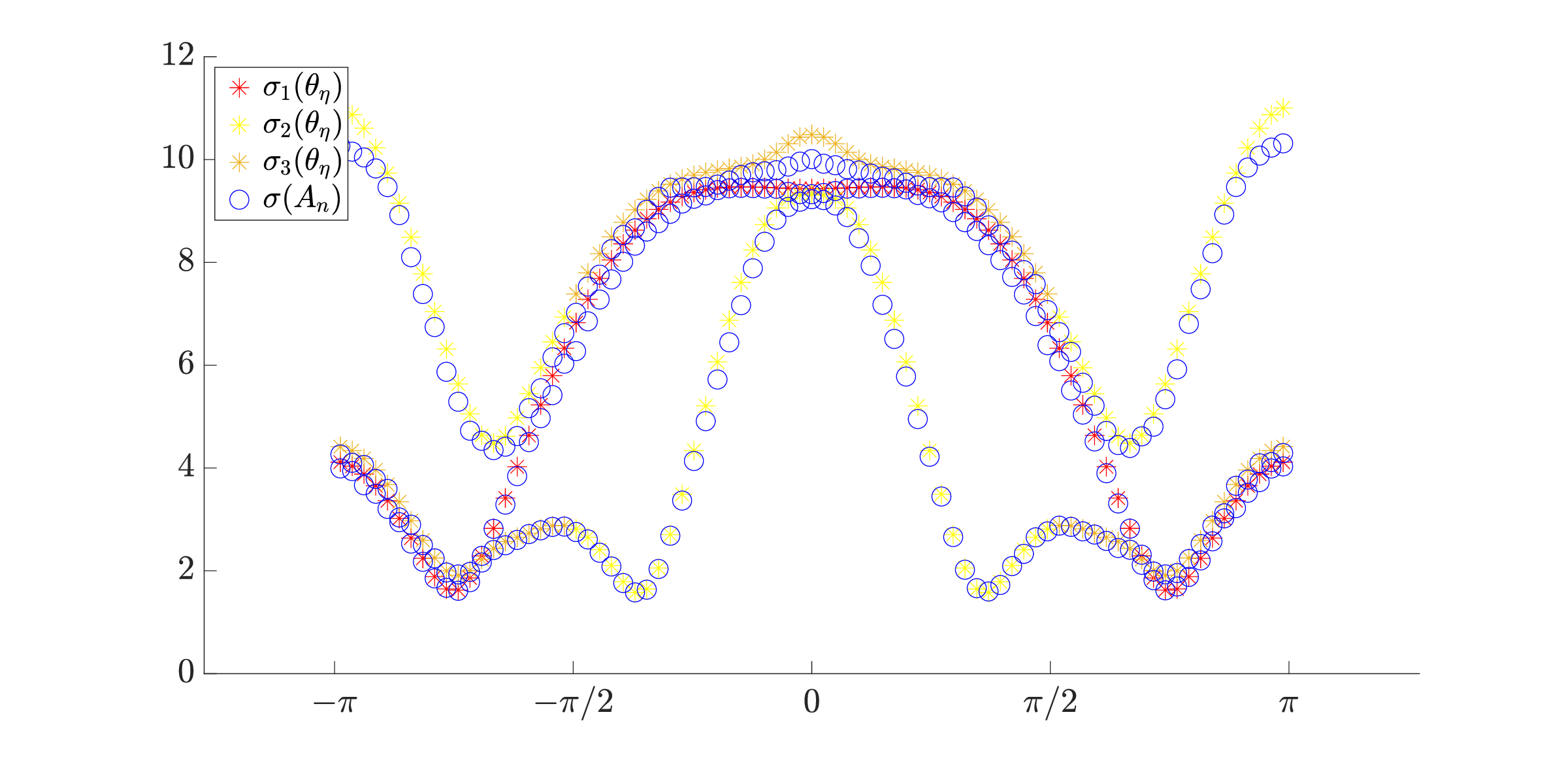}
\end{subfigure}
\end{center}
\caption{Comparison between the singular value functions $\sigma_l(F)$, $l=1,\ldots,3$  and the singular values of $A_n$, for $\nu=2$, $t=2,s=1$, $n_1=\eta$, $n_2=2\eta+\lceil{\sqrt{\eta}\,\rceil}$, $\eta= 25, 49, 81 $.}
\label{fig:nu2t2s1_eig_eta_2etasqrt_254981}
\end{figure}

\paragraph{Group 3. $3\times 3$ case with square matrix-valued $f_{i,j}$}

 For the setting $\nu=3$, $t=s=2$ we consider matrix-valued functions $f_{i,j}$, and  associated Toeplitz matrices $T_{n_i,n_j}(f_{i,j})$, that are related to PDEs discretizations problems. Additionally, we take matrix-valued functions related to symbol-based operators used to handle the numerical solution of the discrete problem. As continuous test problem we consider the 1D constant coefficients second order differential equation given by:
find $u$ such that
\begin{small}
\begin{equation}
\label{FEM_problem}
\begin{cases}
&u''(x)=\psi(x) \quad {\rm on}\, \, (0,1), \\
& u(0)=u(1)=0 ,
\end{cases}
\end{equation}
\end{small}where $\psi(x)\in L^2\left(0,1\right)$.

We select the structures and the functions obtained in the following contexts:

\begin{enumerate}
\item[3.1)]  The Finite Differences (FD) discretization of (\ref{FEM_problem}) which produces a  Toeplitz  matrix generated by $f(\theta)=2-2\cos\theta$ and, equivalently, the related $d=2$ matrix-valued version \cite{huckle}: $${f}^{[d]}(\theta)= \hat{{f}}_0^{[d]}+\hat{{f}}_
{-1}^{[d]}{\rm e}^{-\iota\theta}+\hat{{f}}_1^{[d]}{\rm e}^{\iota\theta},$$
where
\[\hat{{f}}_0^{[d]}=T_d(2-2\cos\theta), \quad \hat{{f}}_{-1}^{[d]}=-e_de_1^T, \quad \hat{{f}}_1^{[d]}=(\hat{{f}}_{-1}^{[d]})^T=-e_1e_d^T.\]
\item[3.2)] The $\mathbb{Q}_{2} $ Lagrangian FEM approximation of (\ref{FEM_problem}), then the related scaled stiffness matrix is  the Toeplitz matrix generated by ${f}_{\mathbb{Q}_2}$ \cite{qp}:
\begin{equation*}
\begin{split}
{f}_{\mathbb{Q}_2}(\theta)=&\frac{1}{3}\left(\begin{bmatrix}
16 & -8 \\
-8 &14
\end{bmatrix}+ \begin{bmatrix}
0 & -8 \\
0 &1
\end{bmatrix}{\rm e}^{\iota\theta}+\begin{bmatrix}
0 & 0\\
-8 &1
\end{bmatrix} {\rm e}^{-{\iota}				\theta}\right)
\end{split}
\end{equation*}
\item [3.3)] The B-Spline discretization of the problem (\ref{FEM_problem}) for different values of degree $p$ and regularity $k$. Here we only consider the pairs
$(p,k)$ equal to $(2,0)$ and $(3,1)$. The obtained structures $A_{p,k}$ are low-rank corrections of the Toeplitz matrices generated by the following functions \cite{tom}:
\begin{align*}
&{f}^{(2,0)}(\theta)=\frac{1}{3}\left(\begin{bmatrix}
4 & -2 \\
-2 &8
\end{bmatrix}+ \begin{bmatrix}
0 & -2 \\
0 &-2
\end{bmatrix}{\rm e}^{\iota\theta}+\begin{bmatrix}
0 & 0\\
-2&-2
\end{bmatrix} {\rm e}^{-{\iota}\theta}\right),\\
&{f}^{(3,1)}(\theta)=\frac{1}{40}\left(\begin{bmatrix}
48 & 0 \\
0&48
\end{bmatrix}+ \begin{bmatrix}
-15& -15 \\
-3 &-15
\end{bmatrix}{\rm e}^{\iota\theta}+\begin{bmatrix}
-15 & -3\\
-15 &-15
\end{bmatrix} {\rm e}^{-{\iota}\theta}\right).
\end{align*}

\item[3.4)] The classical multigrid grid transfer operator obtained following the geometric approach  for finite elements \cite{MR4389580,braess,MR4284081,FRTBS}. In particular the optimality of the algebraic Two Grid method (TGM) is obtained, for degree 2, using  $
P_{n,k}^{{2}}$ associated with

\[{p}_{\mathbb{Q}_2}(\theta)= \begin{bmatrix}
3/4 & 3/8\\
    0 & 1\\
\end{bmatrix}+ \begin{bmatrix}
0& 3/8 \\
0 &0
\end{bmatrix}{\rm e}^{\iota\theta}+\begin{bmatrix}
3/4 & -1/8\\
1 &0
\end{bmatrix} {\rm e}^{-{\iota}\theta}
+\begin{bmatrix}
0 & -1/8\\
0 &0
\end{bmatrix} {\rm e}^{-2{\iota}\theta}.
\]

\end{enumerate}

We then construct the matrix-sequence $\{A_n\}_n$ where
\[A_n=\begin{bmatrix}
T_{n_1}(f_{1,1}) &  T_{n_1,n_2}(f_{1,2}) &  T_{n_1,n_3}(f_{1,3}) \\
T_{n_2,n_1}(f_{2,1}) &  T_{n_2}(f_{2,2})&  T_{n_2,n_3}(f_{2,3})\\
T_{n_3,n_1}(f_{3,1}) &  T_{n_3,n_2}(f_{3,2})&  T_{n_3}(f_{3,3})
\end{bmatrix},\]
with:
\begin{equation}
\label{eq:example_PDE1}
\begin{split}
& f_{1,1}(\theta)= {f}_{\mathbb{Q}_2}(\theta); \quad f_{1,2}(\theta)= f_{2,1}(\theta)= {f}^{(2,0)}(\theta);  \quad f_{1,3}(\theta)= f_{3,1}(\theta)= {f}^{(3,0)}(\theta); \\
& f_{2,2}(\theta)={p}_{\mathbb{Q}_2}^*(\theta){p}_{\mathbb{Q}_2}(\theta)+{p}_{\mathbb{Q}_2}^*(\theta+\pi){p}_{\mathbb{Q}_2}(\theta+\pi);\\
&f_{2,3}(\theta)= f_{3,2}(\theta)= {p}_{\mathbb{Q}_2}(\theta); \quad f_{3,3}(\theta)= {f}^{[d]}(\theta).
\end{split}
\end{equation}

For the selected sequence we prove numerically the distribution result
\[\{A_n\}_n\sim_{\sigma}F,\]
with $F$ being the matrix-valued function with the structure shown in  Theorem \ref{th:general}. Namely,
\[F=
\left[\begin{array}{cc |c |cccc}
f_{1,1} &0       & f_{1,2}& f_{1,3} &0      &0 &0  \\
 0     & f_{1,1} & 0     & 0      & f_{1,3}&0 & 0  \\
\hline
f_{2,1} &0       & f_{2,2}& f_{2,3} &0     &0 & 0  \\
\hline
f_{3,1} &0       & f_{3,2}& f_{3,3} &0     &0 & 0  \\
0      &f_{3,1}  &0      & 0      & f_{3,3}&0&0  \\
0      &       0&  0    &0      &0      & f_{3,3}      & 0  \\
0      &       0&  0    &0      &0      & 0      & f_{3,3}  \\
\end{array}\right].
\]

In Figure \ref{fig:nu3t2s2_sv_204080} we plot the singular values of $A_n$, denoted by ${\rm \sigma(A_n)}$ for different values of $\eta= 20, 40, 80$. We rearrange the singular values into 14 groups corresponding to the  samples of the singular value functions $\sigma_l(F)$, $l=1,\dots, 14$, obtained with \textit{MATLAB} symbolic computation and evaluated on a uniform grid over the interval $[-\pi,\pi]$.
In order to obtain a perfect matching between the size of $A_n$ and the cardinality of the (equispaced) grids where to sample $\sigma_l(F)$, $l=1,\dots, 14$, we follow a specific choice: for $\sigma_i(F)$, $i=1,\dots,9$, we have
$$\theta_{\eta/2}=\left\{-\frac{\pi(\eta/2)}{(\eta/2)+1}+\frac{2j\pi(\eta/2)}{(\eta/2+1)((\eta/2-1)}, \, j=0,\dots, \frac{\eta}{2}-1\right\}$$
having $\eta/2$ points, while for the remaining $\sigma_i(F)$, $i=10,\dots,14$, we use the grid $\theta_{\eta/2-1}$ having cardinality $\eta/2-1$.
Given the asymptotic notions of distribution, any asymptotically uniform grid-sequence is equivalent \cite{au2,au1}. However, for a fixed matrix-size an optimal grid can give a more precise indication of the behavior of the eigenvalues or singular values. For more discussions and numerical evidences regarding optimal grid-sequences, especially in connection with the concepts of momentary symbols the reader is referred to \cite{mom3,mom1,mom2}.

\begin{figure}[htb]
\begin{subfigure}[c]{.50\textwidth}
\includegraphics[width=\textwidth]{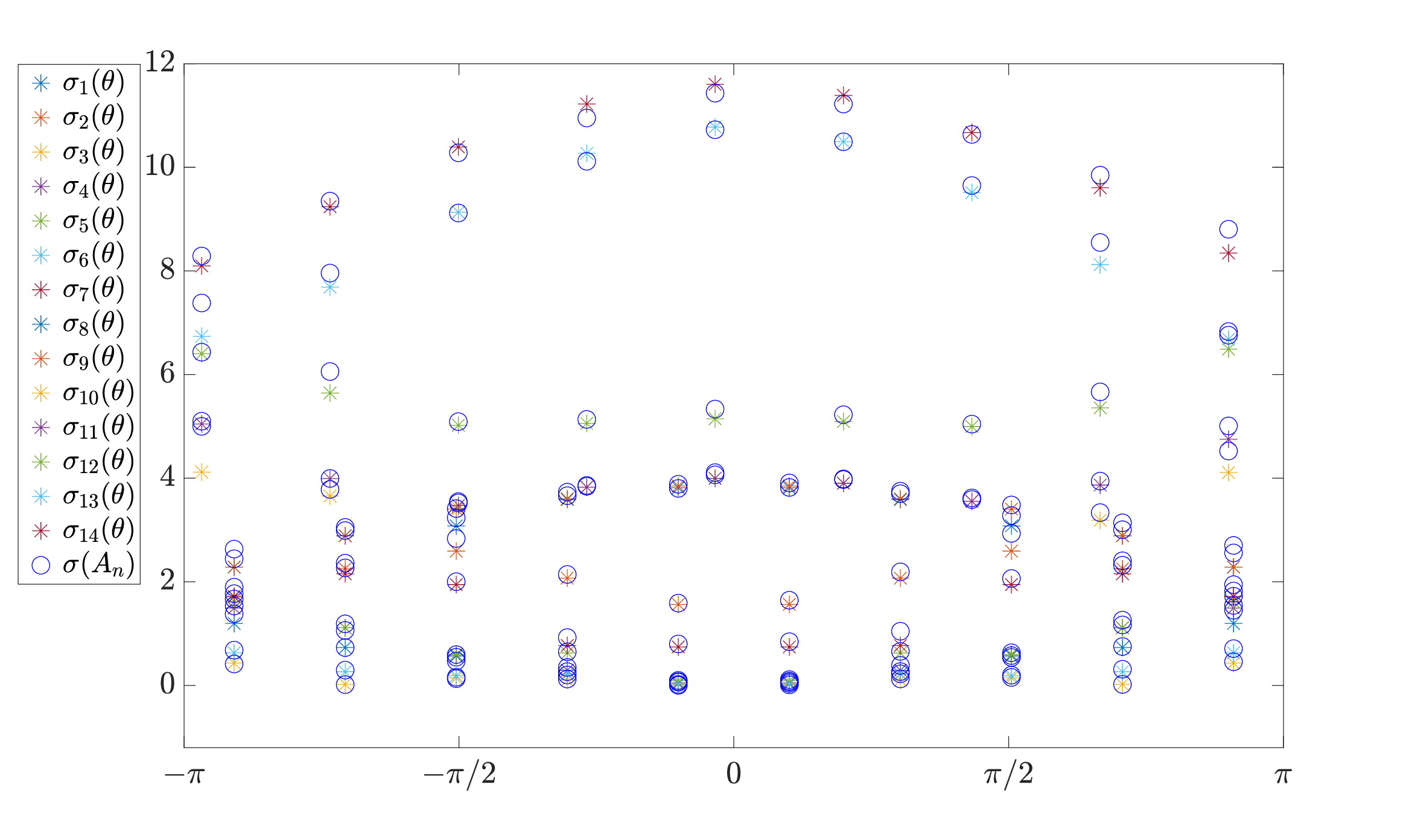}
\end{subfigure}
\begin{subfigure}[c]{.50\textwidth}
\includegraphics[width=\textwidth]{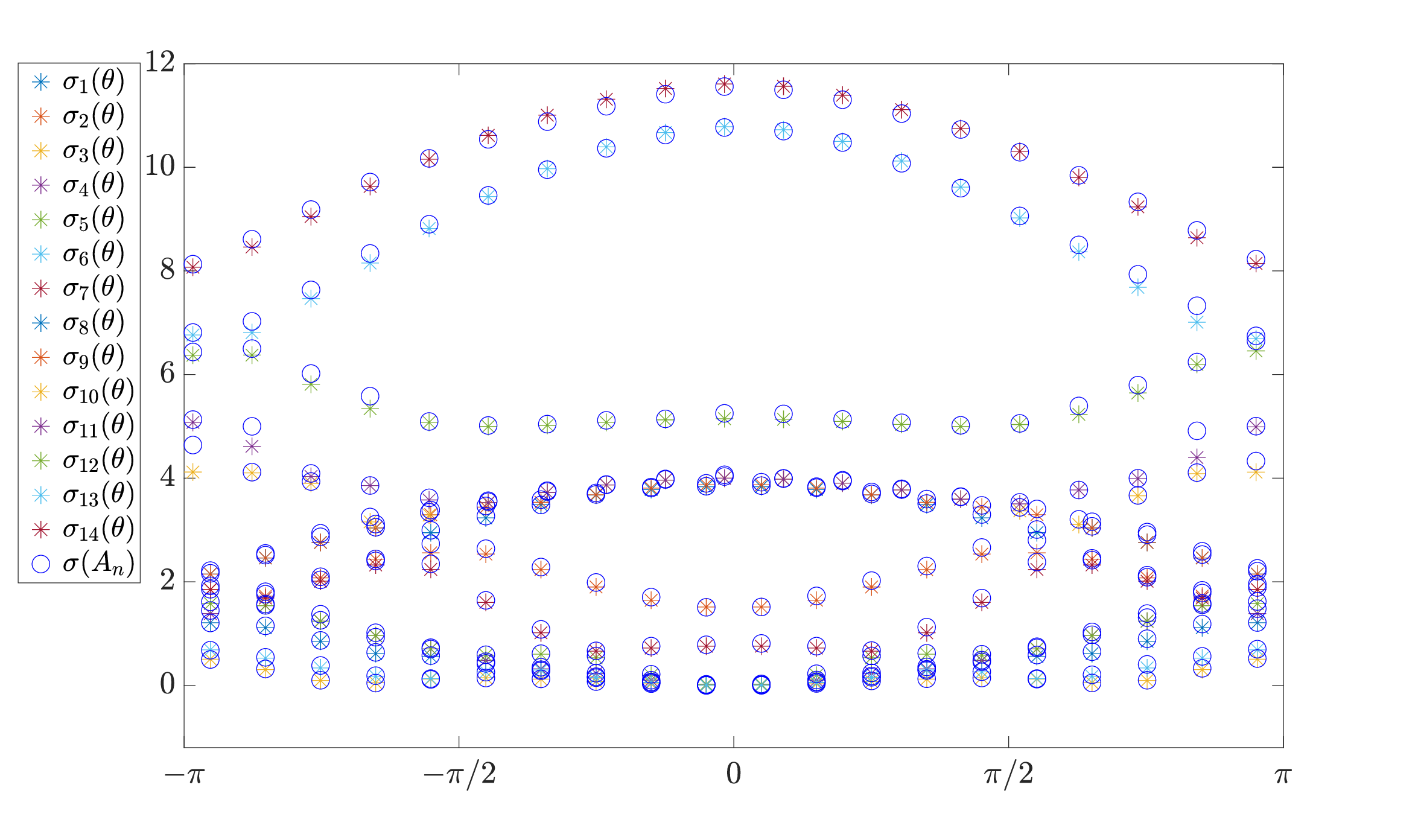}
\end{subfigure}\\
\begin{center}
\begin{subfigure}[c]{.80\textwidth}
\includegraphics[width=\textwidth]{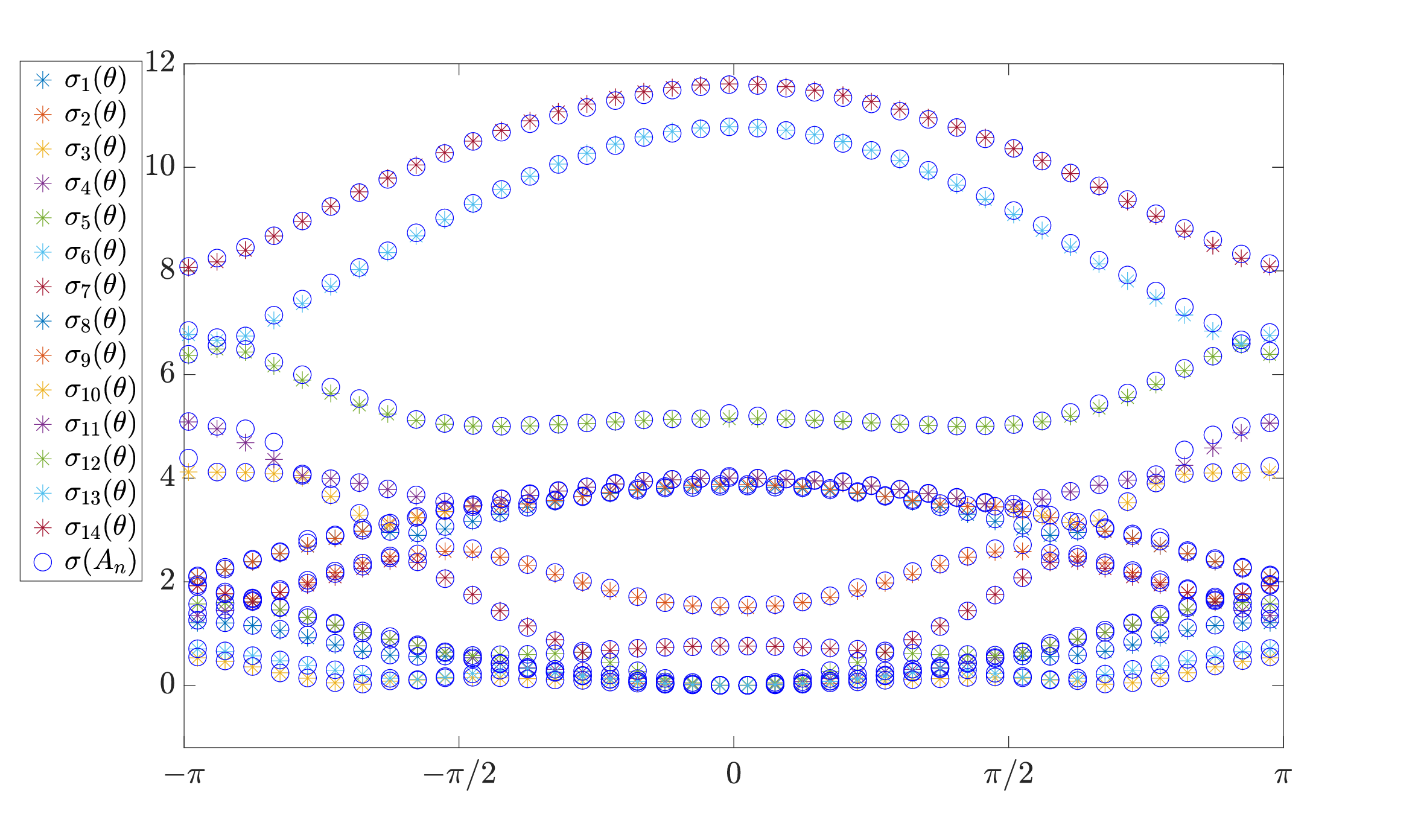}
\end{subfigure}
\end{center}
\caption{Comparison between the singular value functions of $F$, $\sigma_l(\theta)$, $l=1,\ldots,14$  and the singular values of $A_n$, for $\nu=3$, $t,s=2$, $n_1=\eta$, $n_2=\eta/2$, $n_3=2\eta-2$, $\eta= 20, 40, 80 $.}
\label{fig:nu3t2s2_sv_204080}
\end{figure}

We can observe that the singular values mimic, up to outliers, the considered samplings and the approximation improves as $n\rightarrow \infty$, as predicted by the theory.

As an additional experiment we mention that  if we substitute in (\ref{eq:example_PDE1}) the choice $f_{2,3}(\theta)= f_{3,2}(\theta)= {p}_{\mathbb{Q}_2}(\theta)$ with $ f_{3,2}(\theta)= {p}_{\mathbb{Q}_2}(\theta)$ and $f_{2,3}(\theta)=f_{3,2}^*(\theta)$, then we obtain $F=F^*$ and in this case we also have a spectral distribution result:
\[\{A_n\}_n\sim_{\lambda}F.\]
The numerical validation can be observed in Figure \ref{fig:nu3t2s2_eig_204080} with the comparison of the eigenvalues of $A_n$, denoted by ${\rm eig}$, and the eigenvalue functions of $F$, $\lambda_l(\theta)$, $l=1,\dots, 14$, evaluated over an uniform grid over $[-\pi,\pi]$, for different values of $\eta= 20, 40, 80$.

\begin{figure}[htb]
\begin{subfigure}[c]{.50\textwidth}
\includegraphics[width=\textwidth]{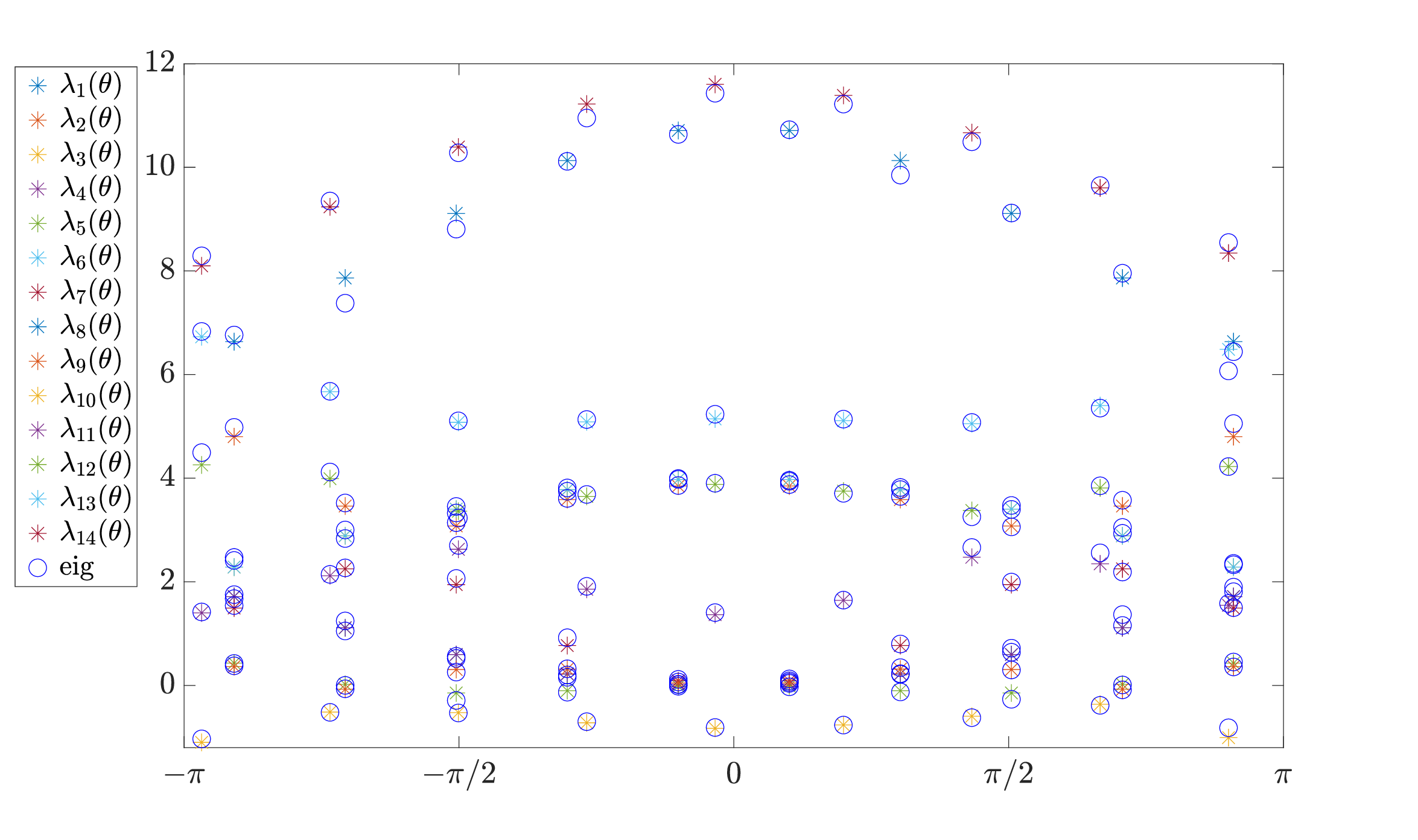}
\end{subfigure}
\begin{subfigure}[c]{.50\textwidth}
\includegraphics[width=\textwidth]{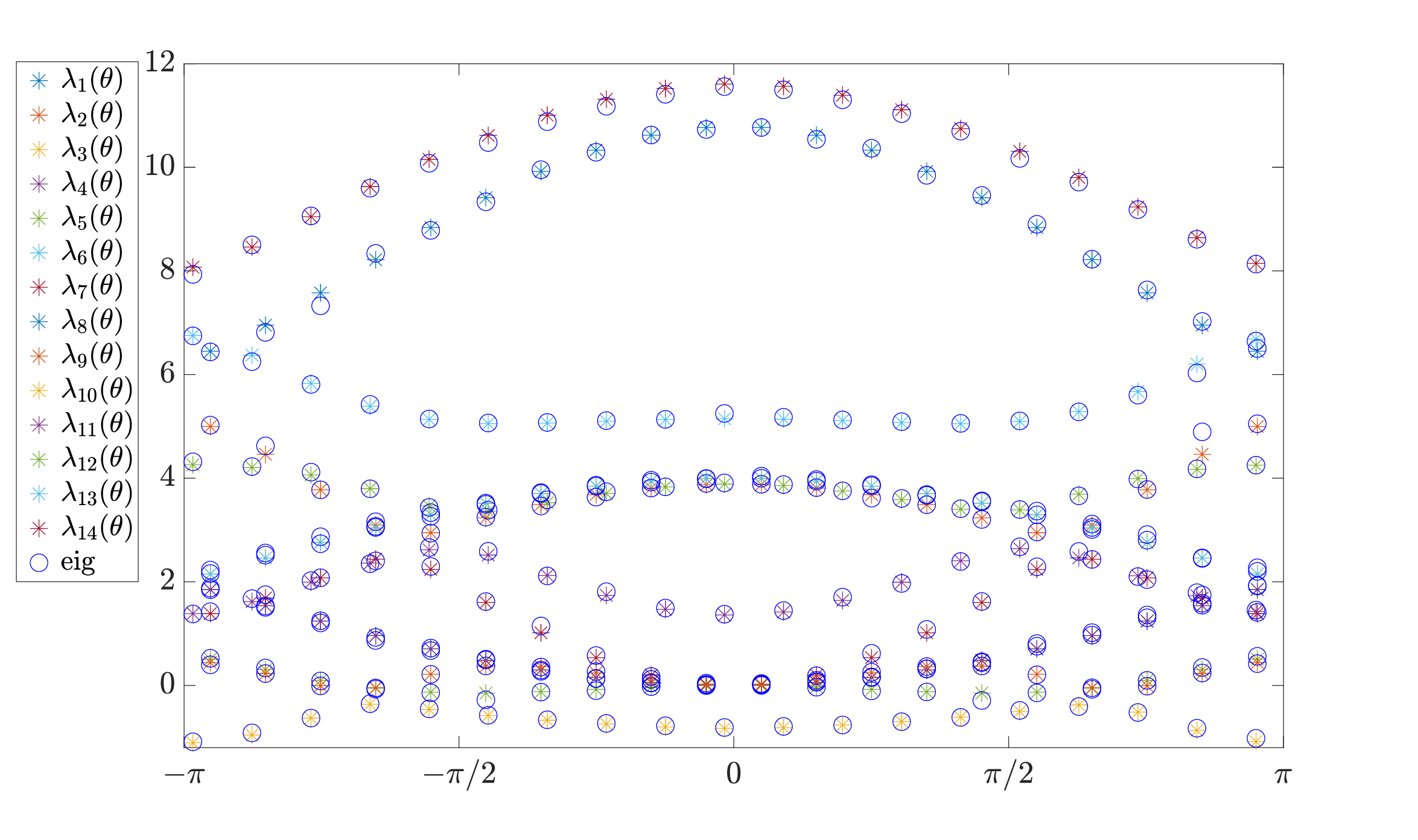}
\end{subfigure}\\
\begin{center}
\begin{subfigure}[c]{.80\textwidth}
\includegraphics[width=\textwidth]{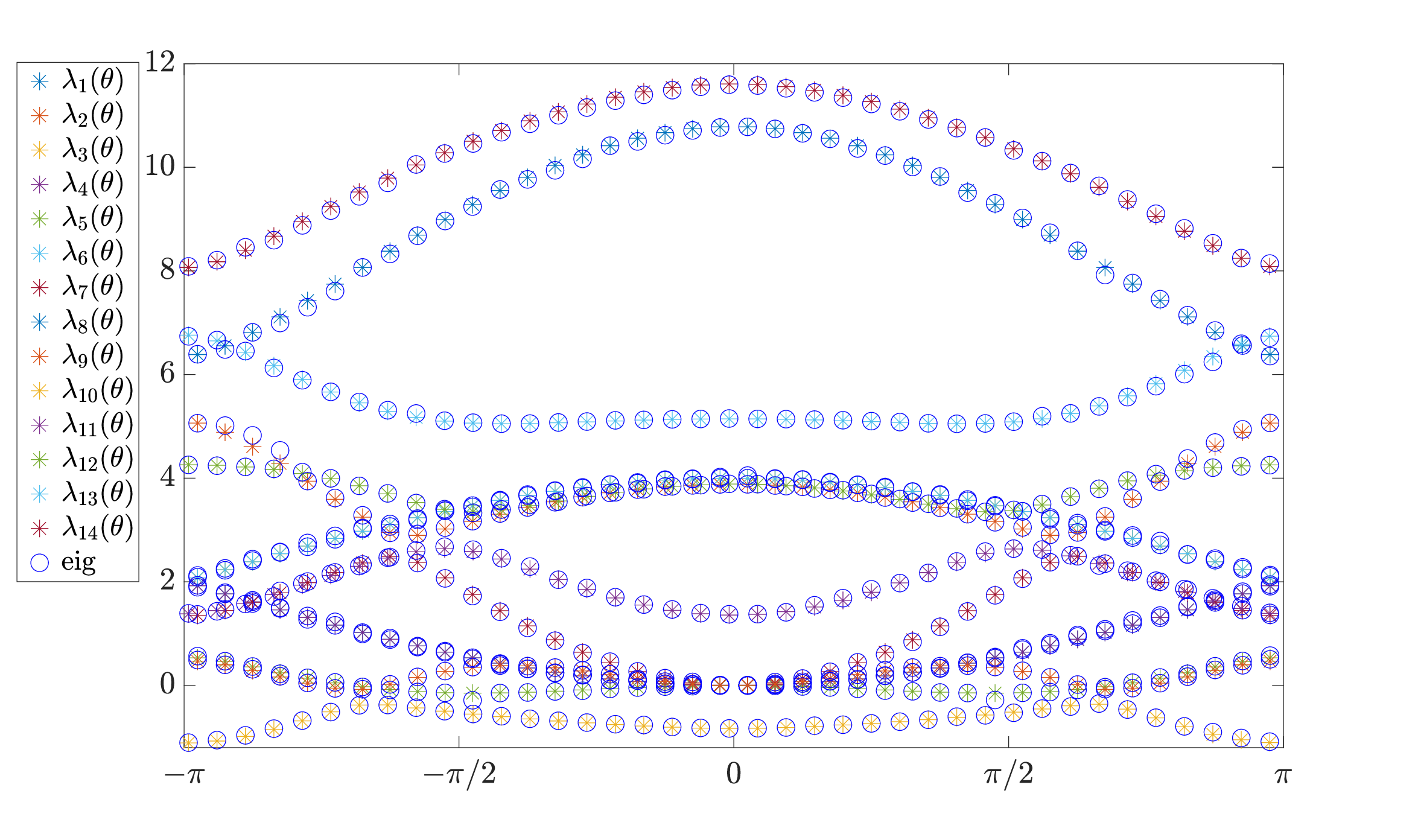}
\end{subfigure}
\end{center}
\caption{Comparison between the eigenvalue functions $\lambda_l(F)$, $l=1,\ldots,14$  and the eigenvalues of $A_n$, for $\nu=3$, $t,s=2$, $n_1=\eta$, $n_2=\eta/2$, $n_3=2\eta-2$, $\eta= 20, 40, 80 $.}
\label{fig:nu3t2s2_eig_204080}
\end{figure}

\section{Conclusions}\label{sec:fin}

We have provided a general framework for studying the asymptotic Weyl distributions in the singular value and eigenvalue sense of matrix-sequences in block form. We have dealt in detail the case where the matrix-sequences associated to the blocks are of block unilevel Toeplitz type. Visualizations and numerical tests have been presented and critically discussed.

There are many directions for expanding the theoretical framework, including investigating the conjecture in Remark  \ref{rmrk:conjecture} and examining also the case where the number of row blocks and column blocks are different, so adding a new level in the rectangular structure of $A_n$.

However, two main steps remain to be performed.
\begin{itemize}
\item The primary open challenge lies in generalizing the results to the multilevel setting, enabling the treatment of approximations for multidimensional differential or fractional differential problems with variable coefficients.
Such a generalization can follow various possible directions and has to include the case where the blocks are of block multilevel Toeplitz type  or derived from block multilevel GLT matrix-sequences. This will involve notational and technical challenges that need to be handled with care and precision.
\item The distributional information can be used to design fast preconditioned Krylov methods for solving the related large linear systems.
In particular, we anticipate that the construction of various preconditioners is somehow implicitly written in the proof of Proposition \ref{prop:equivalence 2}, when considering $\hat{A}_n$. In fact, $\hat{A}_n$ is the important and much sparser part of $A_n$, once the terms associated to zero-distributed matrix sequences have been removed; see \cite{blocking - num}. This simplified structure and the matrix-algebra approximations of its blocks represent the key ingredients for the construction of preconditioneed Krylov solvers and preconditioned multi-iterative methods \cite{MR3689933,dumb,multi-iter}, insuring clustered spectra or clustered singular values, also in a block multilevel Toeplitz/GLT setting.
The precise analysis and the associated numerical performance constitute a worthwhile research direction to be considered in a near future, whose starting point is represented by the current work and by the work in \cite{blocking - num}.
\end{itemize}

\section*{Acknowledgements}

 All the authors are members of ``Gruppo Nazionale per il Calcolo Scientifico" (INdAM-GNCS). The work is partially supported by INdAM - GNCS Project ``Analisi e applicazioni di matrici strutturate (a blocchi)"  CUP E53C23001670001.

The work is funded from the European High-Performance Computing Joint Undertaking  (JU) under grant agreement No 955701. The JU receives support from the European Union’s Horizon 2020 research and innovation programme and Belgium, France, Germany, Switzerland.

The work of Andrea Adriani is supported by MUR Excellence Department Project MatMod@TOV awarded to the Department of Mathematics, University of Rome Tor Vergata, CUP E83C23000330006.

The work of Isabella Furci is supported by $\#$NEXTGENERATIONEU (NGEU) and funded by the Ministry of University and Research (MUR), National Recovery and Resilience Plan (NRRP), project MNESYS (PE0000006) - A Multiscale integrated approach to the study of the nervous system in health and disease (DN. 1553 11.10.2022).

Furthermore Stefano Serra-Capizzano is grateful for the support of the Laboratory of Theory, Economics and Systems – Department of Computer Science at Athens University of Economics and Business and to the "Como Lake center for AstroPhysics” of Insubria University.


\begin{thebibliography}{10}

\bibitem{blocking-gen}
A.~Adriani, I.~Furci, C.~Garoni, and S.~Serra-Capizzano.
\newblock Block multilevel structured matrix-sequences and their spectral and
  singular value canonical distributions: general theory and few emblematic
  applications.
\newblock Preprint {(2025)}.

\bibitem{blocking-conj}
{A.~Adriani, A.J.A.~Schiavoni-Piazza, and S.~Serra-Capizzano.
\newblock Block structures, g.a.c.s. approximation, and distributions.
\newblock {\em  Bolet. Soc. Mat. Mex.} - special Volume in Memory of Prof. Nikolai Vasilevski, in press, 2025.}

\bibitem{gacs}
{A.~Adriani, A.J.A.~Schiavoni-Piazza, S.~Serra-Capizzano, and C. Tablino-Possio.
\newblock Revisiting the notion of approximating class of sequences for handling approximated PDEs on moving or unbounded domains.
\newblock {\em Electron. Trans. Numer. Anal.}, under revision, 2025.}


\bibitem{blocking - num}
{
N.~Barakitis, M. Donatelli, S. Ferri, V. Loi, S.~Serra~Capizzano, and R.L. Sormani.
\newblock Blocking structures, approximation, and preconditioning.
\newblock {\em Preprint}, 2501.14874v1, 2005.}


\bibitem{pre-prequel}
N.~Barakitis, P.~Ferrari, I.~Furci, and S.~Serra-Capizzano.
\newblock An extradimensional approach and distributional results for the case
  of $2\times 2$ block toeplitz structures.
\newblock {Springer Proceedings on Mathematics and Statistics, in press, 2025.}



\bibitem{mom3}
N.~Barakitis, V.~Loi, and S.~Serra-Capizzano.
\newblock A note on eigenvalues and singular values of variable {Toeplitz}
  matrices and matrix-sequences, with application to variable two-step {BDF}
  approximations to parabolic equations.
\newblock {Springer book series ``Trends in Mathematics", in press, 2025.}

\bibitem{rearr}
G.~Barbarino, D.~Bianchi, and C.~Garoni.
\newblock Constructive approach to the monotone rearrangement of functions.
\newblock {\em Expo. Math.}, 40(1):155--175, 2022.

\bibitem{au2}
G.~Barbarino and C.~Garoni.
\newblock An extension of the theory of {GLT} sequences: sampling on
  asymptotically uniform grids.
\newblock {\em Linear Multilinear Algebra}, 71(12):2008--2025, 2023.

\bibitem{GLT-blocks-d-dim}
G.~Barbarino, C.~Garoni, and S.~Serra-Capizzano.
\newblock Block generalized locally {T}oeplitz sequences: theory and
  applications in the multidimensional case.
\newblock {\em Electron. Trans. Numer. Anal.}, 53:113--216, 2020.

\bibitem{GLT-blocks-1-dim}
G.~Barbarino, C.~Garoni, and S.~Serra-Capizzano.
\newblock Block generalized locally {T}oeplitz sequences: theory and
  applications in the unidimensional case.
\newblock {\em Electron. Trans. Numer. Anal.}, 53:28--112, 2020.

\bibitem{EMI}
P.~Benedusi, P.~Ferrari, M.~E. Rognes, and S.~Serra-Capizzano.
\newblock Modeling excitable cells with the {EMI} equations: spectral analysis
  and iterative solution strategy.
\newblock {\em J. Sci. Comput.}, 98(3):Paper No. 58, 23, 2024.

\bibitem{MR4389580}
M.~Bolten, M.~Donatelli, P.~Ferrari, and I.~Furci.
\newblock A symbol-based analysis for multigrid methods for block-circulant and
  block-{T}oeplitz systems.
\newblock {\em SIAM J. Matrix Anal. Appl.}, 43(1):405--438, 2022.

\bibitem{MR4623368}
M.~Bolten, M.~Donatelli, P.~Ferrari, and I.~Furci.
\newblock Symbol based convergence analysis in block multigrid methods with
  applications for {S}tokes problems.
\newblock {\em Appl. Numer. Math.}, 193:109--130, 2023.

\bibitem{mom1}
M.~Bolten, S.-E. Ekstr\"om, I.~Furci, and S.~Serra-Capizzano.
\newblock Toeplitz momentary symbols: definition, results, and limitations in
  the spectral analysis of structured matrices.
\newblock {\em Linear Algebra Appl.}, 651:51--82, 2022.

\bibitem{mom2}
M.~Bolten, S.-E. Ekstr\"om, I.~Furci, and S.~Serra-Capizzano.
\newblock A note on the spectral analysis of matrix sequences via {GLT}
  momentary symbols: from all-at-once solution of parabolic problems to
  distributed fractional order matrices.
\newblock {\em Electron. Trans. Numer. Anal.}, 58:136--163, 2023.

\bibitem{BS}
A.~B\"ottcher and B.~Silbermann.
\newblock {\em Introduction to large truncated {T}oeplitz matrices}.
\newblock Universitext. Springer-Verlag, New York, 1999.

\bibitem{braess}
D.~Braess.
\newblock {\em Finite elements}.
\newblock Cambridge University Press, Cambridge, third edition, 2007.
\newblock Theory, fast solvers, and applications in elasticity theory,
  Translated from the German by Larry L. Schumaker.

\bibitem{missing-data2}
V.~Del~Prete, F.~Di~Benedetto, M.~Donatelli, and S.~Serra-Capizzano.
\newblock Symbol approach in a signal-restoration problem involving block
  {T}oeplitz matrices.
\newblock {\em J. Comput. Appl. Math.}, 272:399--416, 2014.

\bibitem{MR3689933}{dumb}
M.~Donatelli, A.~Dorostkar, M.~Mazza, M.~Neytcheva, and S.~Serra-Capizzano.
\newblock Function-based block multigrid strategy for a two-dimensional linear
  elasticity-type problem.
\newblock {\em Comput. Math. Appl.}, 74(5):1015--1028, 2017.

\bibitem{MR4284081}
M.~Donatelli, P.~Ferrari, I.~Furci, S.~Serra-Capizzano, and D.~Sesana.
\newblock Multigrid methods for block-{T}oeplitz linear systems: convergence
  analysis and applications.
\newblock {\em Numer. Linear Algebra Appl.}, 28(4):Paper No. e2356, 20, 2021.

\bibitem{MR3543002}
A.~Dorostkar, M.~Neytcheva, and S.~Serra-Capizzano.
\newblock Spectral analysis of coupled {PDE}s and of their {S}chur complements
  via generalized locally {T}oeplitz sequences in 2{D}.
\newblock {\em Comput. Methods Appl. Mech. Engrg.}, 309:74--105, 2016.

\bibitem{dumb}
M.~Dumbser, F.~Fambri, I.~Furci, M.~Mazza, S.~Serra-Capizzano, and M.~Tavelli.
\newblock Staggered discontinuous {G}alerkin methods for the incompressible
  {N}avier-{S}tokes equations: spectral analysis and computational results.
\newblock {\em Numer. Linear Algebra Appl.}, 25(5):e2151, 31, 2018.

\bibitem{nonunique}
S.-E. Ekstr\"om and S.~Serra-Capizzano.
\newblock Eigenvalues and eigenvectors of banded {T}oeplitz matrices and the
  related symbols.
\newblock {\em Numer. Linear Algebra Appl.}, 25(5):e2137, 17, 2018.

\bibitem{MR1740439}
D.~Fasino and P.~Tilli.
\newblock Spectral clustering properties of block multilevel {H}ankel matrices.
\newblock {\em Linear Algebra Appl.}, 306(1-3):155--163, 2000.

\bibitem{FRTBS}
P.~Ferrari, R.~I. Rahla, C.~Tablino-Possio, S.~Belhaj, and S.~Serra-Capizzano.
\newblock Multigrid for $\mathbb{Q}_{k}$ finite element matrices using a
  (block) {T}oeplitz symbol approach.
\newblock {\em Mathematics}, 8(5), 2020.

\bibitem{GSI}
C.~Garoni and S.~Serra-Capizzano.
\newblock {\em Generalized locally {T}oeplitz sequences: theory and
  applications. Vol. I}.
\newblock Springer, Cham 1-312, 2017.

\bibitem{GSII}
C.~Garoni and S.~Serra-Capizzano.
\newblock {\em Generalized locally {T}oeplitz sequences: theory and
  applications. Vol. II}.
\newblock Springer, Cham 1-194, 2018.

\bibitem{qp}
C.~Garoni, S.~Serra-Capizzano, and D.~Sesana.
\newblock Spectral analysis and spectral symbol of {$d$}-variate {$\Bbb{Q}_p$}
  {L}agrangian {FEM} stiffness matrices.
\newblock {\em SIAM J. Matrix Anal. Appl.}, 36(3):1100--1128, 2015.

\bibitem{tom}
C.~Garoni, H.~Speleers, S.-E. Ekstr\"om, A.~Reali, S.~Serra-Capizzano, and
  T.~J.~R. Hughes.
\newblock Symbol-based analysis of finite element and isogeometric {B}-spline
  discretizations of eigenvalue problems: exposition and review.
\newblock {\em Arch. Comput. Methods Eng.}, 26(5):1639--1690, 2019.

\bibitem{block-ext2}
G.~H. Golub and C.~Greif.
\newblock On solving block-structured indefinite linear systems.
\newblock {\em SIAM J. Sci. Comput.}, 24(6):2076--2092, 2003.

\bibitem{block-ext5}
T.~Gong, W.~Zhang, and Y.~Chen.
\newblock Uncovering block structures in large rectangular matrices.
\newblock {\em J. Multivariate Anal.}, 198:Paper No. 105211, 19, 2023.

\bibitem{MR0890515}
U.~Grenander and G.~Szeg\"o.
\newblock {\em Toeplitz forms and their applications}.
\newblock Chelsea Publishing Co., New York, second edition, 1984.

\bibitem{block-ext8}
L.~P. Huang and S.~W. Zou.
\newblock Geometry of rectangular block triangular matrices.
\newblock {\em Acta Math. Sin. (Engl. Ser.)}, 25(12):2035--2054, 2009.

\bibitem{huckle}
T.~K. Huckle.
\newblock Compact {F}ourier analysis for designing multigrid methods.
\newblock {\em SIAM J. Sci. Comput.}, 31(1):644--666, 2008.

\bibitem{block-ext1}
Z.~Iqbal, S.~Nooshabadi, I.~Yamazaki, S.~Tomov, and J.~Dongarra.
\newblock Exploiting block structures of {KKT} matrices for efficient solution
  of convex optimization problems.
\newblock {\em IEEE Access}, 9:116604--116611, 2021.

\bibitem{block-ext3}
S.~Lungten, W.~H.~A. Schilders, and J.~M.~L. Maubach.
\newblock Sparse block factorization of saddle point matrices.
\newblock {\em Linear Algebra Appl.}, 502:214--242, 2016.

\bibitem{MR3904142}
M.~Mazza, A.~Ratnani, and S.~Serra-Capizzano.
\newblock Spectral analysis and spectral symbol for the 2{D} curl-curl
  (stabilized) operator with applications to the related iterative solutions.
\newblock {\em Math. Comp.}, 88(317):1155--1188, 2019.

\bibitem{block-ext4}
J.~Nocedal and S.~J. Wright.
\newblock {\em Numerical optimization}.
\newblock Springer Series in Operations Research and Financial Engineering.
  Springer, New York, second edition, 2006.

\bibitem{block-ext6}
D.~S. Parker.
\newblock A randomizing butterfly transformation useful in block matrix
  computations.
\newblock {\em Technical Report CSD-950024, Computer Science Department,
  University of California}, 1995.

\bibitem{multi-iter}
{
S.~Serra~Capizzano.
\newblock Multi-iterative methods.
\newblock {\em Comput. Math.  Appl.}, 26(4):65--87, 1993.}


\bibitem{au1}
S.~Serra~Capizzano and C.~Tablino~Possio.
\newblock Analysis of preconditioning strategies for collocation linear
  systems.
\newblock {\em Linear Algebra Appl.}, 369:41--75, 2003.

\bibitem{MR1705731}
S.~Serra~Capizzano and P.~Tilli.
\newblock Extreme singular values and eigenvalues of non-{H}ermitian block
  {T}oeplitz matrices.
\newblock {\em J. Comput. Appl. Math.}, 108(1-2):113--130, 1999.

\bibitem{MR1671591}
P.~Tilli.
\newblock A note on the spectral distribution of {T}oeplitz matrices.
\newblock {\em Linear and Multilinear Algebra}, 45(2-3):147--159, 1998.

\bibitem{TyZ}
E.~Tyrtyshnikov and N.~Zamarashkin.
\newblock Spectra of multilevel {T}oeplitz matrices: advanced theory via simple
  matrix relationships.
\newblock {\em Linear Algebra Appl.}, 270:15--27, 1998.

\bibitem{block-ext7}
C.~Vermeersch and B.~De~Moor.
\newblock Recursive algorithms to update a numerical basis matrix of the null
  space of the block row, (banded) block {T}oeplitz, and block {M}acaulay
  matrix.
\newblock {\em SIAM J. Sci. Comput.}, 45(2):A596--A620, 2023.

\end{thebibliography}

\end{document}